\documentclass[11pt,a4paper]{amsart}
\usepackage[T1]{fontenc}
\usepackage{latexsym}
\usepackage{amsfonts}
\usepackage{amssymb}
\usepackage{url}
\usepackage{mathrsfs}
\usepackage{amscd}

\usepackage{fullpage}

\newcommand{\ra}{\longrightarrow}

\newcommand{\Fr}{{\rm Fr}}
\newcommand{\Tr}{\mathrm{Tr}}

\newcommand{\Gal}{{\rm Gal}}

\newcommand{\mt}{\mathcal}
\newcommand{\nsp}{{\rm N}_{\rm Spin}}
\newcommand{\disc}{{\rm disc}}

\newcommand{\phit}{\tilde{\varphi}}

\newcommand{\Aa}{\mathbf{A}}
\newcommand{\C}{\mathbf{C}}

\newcommand{\A}{\mathbf{A}}

\newcommand{\Z}{\mathbf{Z}}
\newcommand{\PP}{\mathbb{P}}
\newcommand{\R}{\mathbf{R}}

\newcommand{\Q}{\mathbf{Q}}
\newcommand{\F}{\mathbf{F}}
\newcommand{\G}{\mathbf{G}}

\newcommand{\eps}{\varepsilon}
\renewcommand{\epsilon}{\varepsilon}

\theoremstyle{plain}
\newtheorem{theorem}{Theorem}

\newtheorem{lemme}[theorem]{Lemma}
\newtheorem{corollaire}[theorem]{Corollary}

\newtheorem{proposition}[theorem]{Proposition}

\theoremstyle{remark}
\newtheorem*{remark}{Remark}
\newtheorem*{remarks}{Remarks}

\theoremstyle{definition}
\newtheorem{definition}[theorem]{Definition}

\begin{document}

%%%%%%%%%%%%%%%%%%%%%%%%%%%%%%%%%%%%%%%%%%%%%%%%%%%%%%%%%%%%%%%%%%%%%%%%%%%%%%%%%%%%%%%%%%%%%%%%%%%%%%%%%%%%%
%%%%%%%%%%%%%%%%%%%%%%%%%%%%%%%%%%%%%%%%%%%%%%%%%%%%%%%%%%%%%%%%%%%%%%%%%%%%%%%%%%%%%%%%%%%%%%%%%%%%%%%%%%%%%%%%

\title[$L$-functions of elliptic curves over function fields]{Maximal Galois group of\\ $L$-functions of elliptic curves} 

%%%%%%%%%%%%%%%%%%%%%%%%%%%%%%%%%%%%%%%%%%%%%%%%%%%%%%%%%%%%%%%%%%%%%%%%%%%%%%%%%%%%%%%%%%%%%%%%%%%%%%%%%%%%%
%%%%%%%%%%%%%%%%%%%%%%%%%%%%%%%%%%%%%%%%%%%%%%%%%%%%%%%%%%%%%%%%%%%%%%%%%%%%%%%%%%%%%%%%%%%%%%%%%%%%%%%%%%%%%%%%

\author{F. Jouve}
\address{Dept. of Mathematics\\
The University of Texas at Austin\\
1 University Station C1200\\
Austin, TX, 78712\\
USA.}
\email{jouve@math.utexas.edu}

\subjclass[2000]{11N36, 11G25 (Primary); 11E08, 14D10, 11C08 (Secondary)}
\keywords{$L$-functions of varieties over function fields, monodromy of $\ell$-adic sheaves, large sieve, polynomials and orthogonal matrices over finite fields}

\begin{abstract}
We give a quantitative version of a result due to N. Katz about $L$-functions of elliptic curves over function fields over finite fields. Roughly speaking, Katz's Theorem states that, on average over a suitably chosen algebraic family, the $L$-function of an elliptic curve over a function field becomes ``as irreducible as possible'' when seen as a polynomial with rational coefficients, as the cardinality of the field of constants grows. A quantitative refinement is obtained as a corollary of our main result which gives an estimate for the proportion of elliptic curves studied whose $L$-functions have ``maximal'' Galois group . To do so we make use of E. Kowalski's idea to apply large sieve methods in algebro-geometric contexts. Besides large sieve techniques, we use results of C. Hall on finite orthogonal monodromy and previous work of the author on orthogonal groups over finite fields.
\end{abstract}

\maketitle

\section*{Introduction}
 %In~\cite{KoCrelle} (see also~\cite{KoIMRN}), Kowalski combines arguments of Chavdarov (see~\cite{Ch}) and large sieve techniques to get a quantitative version of Chavdarov's Theorem. The result asserts that if we vary a curve $C/\F_{q^n}$ among a suitably chosen algebraic family, the proportion of curves whose zeta function has a numerator with maximal Galois group over $\Q$ approaches $1$ as $n$ tends to infinity. An instance of Kowalski's refinement is~\cite[Th. 8.15]{KoLS} (see also~\cite[Remark 8.16]{KoLS}): if $f\in\F_q[x]$ (with $q$ assumed to be coprime to $2$) is a polynomial of degree $2g\geqslant 2$ and $U$ is the complement in $\A^1$ of the zeros of $f$, then we may consider the cover $C\ra U$ whose fibers are hyperelliptic curves with genus $g$ and where the affine part of the curve $C_u$ over $u$ has equation $y^2=f(x)(x-u)$. Let $P_u$ be the numerator of the zeta function of $C_u$ and let $N(f,q)$ be the number of $u\in U(\F_q)$ such that $P_u$ is reducible or has ``small'' Galois group over $\Q$. Then we have
% $$
% N(f,q)\leqslant C g^2q^{1-\gamma}\log q\,,
% $$
% with an absolute constant $C$ and where we can choose $\gamma^{-1}=4g^2+2g+4$. Note that estimate is uniform with respect to the common genus of the fibers $g$.
% \par\medskip
 
% We can ask the same question of ``generic maximality of the Galois group'' in other algebro-geometric contexts. For instance, 

 In~\cite{Ka}, Katz studies the irreducibility of the $L$-function of an elliptic curve over a function field in one variable over $\F_q$, when varying the curve among the elements of an algebraic family. A very down-to-earth instance of Katz's result goes as follows (see~\cite[Section 2]{Ka}): let $Twist_d$ be the subspace of $\A^d_{\F_q}$ (where $q$ is still assumed to be odd) consisting of monic polynomials of degree $d\geqslant 3$ for which $f(0)f(1)\disc(f)$ is invertible. For an $\F_q$-rational point $f\in Twist_d(\F_q)$, we consider the affine curve $U$ which is the complement in $\A^1_{\F_q}$ of $0,1$ and the zeros of $f$. We have the family of twisted Legendre elliptic curves $(E_f)$ over $U$ with affine part:
 $$
 y^2=f(\lambda)x(x-1)(x-\lambda)\,.
 $$
 
  From a fundamental Theorem of Grothendieck we know that the $L$-function $L(E_f/\F_q(t);T)$ attached to $E_f/\F_q(t)$ is apolynomial in $T$ with coefficients in $\Z$. In the above setting, Katz's result asserts that as $q$ grows the proportion of $f$ in $Twist_d(\F_q)$ such that the reduced $L$-function of $E_f$ (which is the quotient of $L(E_f/\F_q(t);T)$ by ``trivial factors'' we will make precise in Section~\ref{quant-katz}) is a $\Q$-irreducible polynomial, tends to $1$. That kind of question, of course, is linked to the function field version of the Birch-Swinnerton-Dyer conjecture. Indeed getting information on the irreducibility of the $L$-function of a given elliptic curve over $\F_q(t)$ tells us, a fortiori, about the order of vanishing of that function at $1$.
 \par
 
 In this paper we give a quantitative refinement of Katz's result. There are essentially two types of ingredients that come into play in our proof. First, for an odd prime $\ell$ invertible in the base field, the $\ell$-torsion of the elliptic curves we consider gives rise to a family of $\F_\ell$-sheaves, the monodromy of which is controlled by results of C. Hall~\cite{H}. From loc. cit., we know that the finite monodromy groups arising are big subgroups of certain orthogonal groups and that imposes severe restrictions on the representation theory of those groups. The second main feature of our work consists in local computations for sets of matrices or polynomials over finite fields. Here we make a crucial use of results of~\cite{J}. In both cases one of the main issues comes from the lack of good topological properties of the orthogonal group (seen as an algebraic group over the rationals). As a matter of comparison, Kowalski in~\cite{KoCrelle} establishes a similar quantitative refinement for a result of Chavdarov (see~\cite{Ch}). However these authors work in an algebro-geometric framework where $\F_\ell$-sheaves with symplectic monodromy naturally arise. Therefore most of the main arguments used by Kowalski in~\cite{KoCrelle} or Chavdarov in~\cite{Ch} cannot be applied or even adapted to our case because orthogonal groups do not share some nice properties that hold for the symplectic group (${\bf O}(N)$ is not connected ans ${\bf SO}(N)$ is not simply connected). 
 \par
 Our last task is then to put things together in a large sieve framework we will make precise.

%  In order to do so, we need first to detect the differences between the question that Chavdarov studies in~\cite{Ch} and the one that Katz answers in~\cite{Ka}. We may then adapt the additional arguments of Kowalski to Katz's situation. In so doing, we point out that a crucial difference between the two questions is of a geometric nature. Indeed, contrary to the case studied by Chavdarov and Kowalski where the finite monodromy groups that naturally appear are symplectic groups, the question of the irreducibility of $L$-functions of elliptic curves over function fields over $\F_q$ involves orthogonal groups. This distinction leads to many nontrivial complications. The point is that the orthogonal group (as an algebraic group over $\F_q$) does not share the same properties as the symplectic group (for instance the orthogonal group is neither a connected nor a simply connected linear algebraic group).
 \par
 \medskip
  The paper is divided in the following sections: first, following~\cite{H}, we give the precise context and recall the results we need to state a sample of our main Theorem. In Section~\ref{ls} we set up the large sieve framework that enables us to gather our different pieces of data. In that section we state a large sieve inequality which is a crucial tool to obtain the kind of quantitative information we want. Finally in Section~\ref{main-th}, we recall and make a few results from~\cite{J} more precise, so that putting things together yields a proof of our main result. We conclude by discussing uniformity issues in our result and we give a uniform estimate (in a sense we will make precise) in the case where $E$ is a Legendre curve.
 
 \par
\medskip
\textbf{Notation.}
 As usual the cardinality of a finite set $X$ will be denoted $|X|$. If $G$ is a group and $S$ is a conjugacy invariant subset of $G$, we will denote by $S^\sharp$ the set of conjugacy classes of elements of $S$ under $G$. If $G$ is an abelian group, $\hat{G}$ will denote its group of characters.
 \par
 The biggest integer smaller than a real number $x$ will be denoted $\lfloor x\rfloor$. If $p$ is a prime number and $n$ is an integer, $v_p(n)$ will denote the $p$-adic valuation of $n$. Similarly, if $m(t)$ (resp. $P(t)$) is a polynomial (resp. an irreducible polynomial) in $k[t]$ (for some field $k$) then we will denote by ${\rm ord}_P(m)$ the exponent of $P(t)$ appearing in the decomposition of $m(t)$ as a product of irreducible polynomials.
 \par
  We will make use of the Vinogradov notation $f\ll g$ or of the Landau notation $f=O(g)$ when there exists a strictly positive constant $C$ such that $|f(x)|\leqslant C g(x)$ for any $x$ in a common subset of the domains of the functions $f$ and $g$. If $\mt{P}_\kappa(S)$ is a property that a set $S$ may satisfy depending on the value taken by a real parameter $\kappa$, we will write ``$\mt{P}_\kappa(S)$ holds for $\kappa\geqslant\kappa_0(S)$'' if the property $\mt{P}_\kappa(S)$ is true for a large enough value of $\kappa$ that depends only on the choice of $S$. 
  \par
  For two $n$-tuples $(x_1,\ldots,x_n)$ and $(y_1,\ldots,y_n)$ we will use the Kronecker symbol $$\delta\bigl((x_1,\ldots,x_n),(y_1,\ldots,y_n)\bigr)\,,$$ which equals $1$ if $x_i=y_i$ for all $1\leqslant i\leqslant n$ and $0$ otherwise.
  \par
  Finally we will denote by $\disc(f)$ the discriminant of a polynomial $f$ (in the sense of e.g.~\cite[p. 204]{LaA}) and by $\disc(Q)$ the discriminant of a quadratic form $Q$ (i.e. the determinant of a Gram matrix for $Q$).

 \par
\medskip
\textbf{Acknowledgements.} The author would like to thank E. Kowalski for introducing him to his deep generalization of the large sieve as well as C. Hall and F. Rodriguez-Villegas for useful discussions.

 \section{A quantitative version of Katz's Theorem} \label{quant-katz}
 
  For $q$ a fixed power of a fixed odd prime number $p$, let $E/\F_q(t)$ be an elliptic curve over the rational function field $K=\F_q(t)$ with non constant $j$-invariant. The $L$-function of $E/K$ is well-known to be equal to a reversed characteristic polynomial in the variable $T$:
  $$
  \det(1-qTA)\,,
  $$
  where $A\in O(N,\R)$, the orthogonal group for a real quadratic space of dimension $N$. As mentioned in the introduction we know that such a polynomial has integral coefficients. Following Katz (see~\cite{Ka}), we prefer studying the \emph{unitarized} $L$-function of $E/K$:
  $$
  L_u(E/K;T)=L(E/K;T/q)=\det(1-TA)\,,
  $$
  which is a polynomial with coefficients in $\Z[1/q]$.
  
\par
 The question of the irreducibility of $L_u(E/K;T)$ seen as a $\Q$-polynomial is only relevant if there is no a priori imposed root for that polynomial. But, as $L_u(E/K;T)$ coincides with the (reversed) characteristic polynomial of an orthogonal matrix $A$ with size $N\times N$, it satisfies the well known functional equation
 \begin{equation}\label{eq_fonc}
\det(1-TA)=T^N\det(-A)\det(1-T^{-1}A)\,. 
 \end{equation}

 As a consequence we will study the irreducibility (or indeed the maximality of the Galois group) of what Katz calls the \emph{reduced} $L$-function: 
\begin{equation}\label{red-char}
 L_{\rm red}(E/K;T)=\det(1-TA)_{\rm red}=\begin{cases} &\det(1-TA)/(1-\det(A)T)\,,\,\text{if}\,\,N\,\text{is odd}\,,\\
                                &\det(1-TA)/(1-T^2)\,,\,\text{if}\,\,N\,\text{is even and}\,\,\det(A)=-1\,,\\
                                 &\det(1-TA)\,,\,\text{otherwise}\,.
                 \end{cases}
 \end{equation}

\subsection{Quadratic twists and a particular cohomology space}\label{quad-twist}

  Notice that instead of the above piece of data $E/\F_q(t)$, we can suppose equivalently that we are given an elliptic fibration
  $$
  \mt{E}\ra \PP^1\,.
  $$
  
   Let $f\in K^\times/(K^\times)^2$ (i.e. we fix a class of $K^\times$ modulo the nonzero squares of $K$). We denote by $E_f$ the quadratic twist of $E$ by $f$: this is the elliptic curve such that $E_f\simeq E$ over a quadratic extension of $K$ but not over $K$ itself. For each such $f$ we have an elliptic pencil
   $$
   \mt{E}_f\ra \PP^1\,.
   $$ 

 The group of sections $\mt{E}_f(\PP^1)$ for that fibration can be identified with the Mordell-Weil group $E_f(K)$ and the multiplication by $\ell$ morphism
 $$
 E_f(K)\stackrel{\times \ell}{\ra} E_f(K)\,,
 $$
gives rise, for each $\ell$ invertible in $\F_q$, to  the isogeny of group schemes
$$
\mt{E}^{{\rm lisse}}_f\stackrel{\times \ell}{\ra} \mt{E}^{{\rm lisse}}_f\,,
$$
where $\mt{E}^{{\rm lisse}}_f$ is the smooth part of the N\'eron model $\mt{E}_f\ra \PP^1$ of $E/K$. The kernel $\mt{E}_{f,\ell}$ of that last arrow is an \'etale group scheme over $\PP^1$ and the fiber of $\mt{E}_{f,\ell}$ over a geometric generic point of $\PP^1$ is $E_f[\ell]$, the $\ell$-torsion of $E_f(K)$ (it might only be a subspace of it though, if we do not assume the geometric point we consider to be generic).
\par
 Now let $\mt{V}_{f,\ell}$ denote the \'etale cohomology group $H^1(\PP^1\times \overline{\F_q},\mt{E}_{f,\ell})$. There is an additional quadratic structure on $\mt{V}_{f,\ell}$. Indeed the usual Weil pairing on $E_f[\ell]\times E_f[\ell]$ extends to a non-degenerate alternating pairing
 $$
 \mt{E}_{f,\ell}\times \mt{E}_{f,\ell}\ra \F_\ell(1)\,,
 $$
where $\F_\ell(1)$ denotes the usual Tate twist. Poincar\'e duality yields a non degenerate symmetric pairing
$$
\mt{V}_{f,\ell}\times \mt{V}_{f,\ell}\ra H^2(\PP^1\times \overline{\F_q},\F_\ell(1))\,.
$$

 That last cohomology space being isomorphic to $\F_\ell$, we end up with a non-degenerate symmetric bilinear form on $\mt{V}_{f,\ell}$ for which the endomorphism $\Fr_q$ (induced on $\mt{V}_{f,\ell}$ by the global Frobenius on $\PP^1$) is a conjugacy class of isometries (i.e. of elements of the associated orthogonal group $O(\mt{V}_{f,\ell})$).

\subsection{A $1$-parameter family of quadratic twists}
 With notation as above, we work under the assumptions of~\cite[Section 6.2]{H}: we suppose that $\mt{E}\ra \PP^1$ has at least one fiber of multiplicative reduction away from $\infty$ and we fix a nonzero polynomial $m\in\F_q[t]$ which vanishes at  (at least) one point of the locus of multiplicative reduction so that, for any $f$, the fibration $\mt{E}_f\ra \PP^1$ also has at least one fiber of multiplicative reduction away from $\infty$. Now consider for every integer $d\geqslant 1$
 $$
 F_d=\{f\in\overline{\F_q}[t]\mid f\text{ is squarefree }, \deg(f)=d, {\rm gcd}(f,m)=1\}\,.
 $$

 A remarkable property that $F_d$ satisfies is that the degree of the $L$-function of the twisted curve $E_f/\F_{q^n}$ does neither depend on the choice of $f\in F_d(\F_{q^n})$ nor on $n$ but only on $d$ (see~\cite[beginning of $6.2$]{H}). The degree $N$ of $L(E_f/\F_{q^n}(t);T)$ is explicitly given in Lemma $6.2$ of loc. cit. by
\begin{equation} \label{N-M-A}
N=\deg(M_f)+2\deg(A_f)-4\,,
\end{equation}
where $M_f$ (resp. $A_f$) is the divisor of multiplicative (resp. additive) reduction of $\mt{E}_f\ra \PP^1$.
 
\par
\medskip
 Instead of considering the whole variety $F_d$, we restrict ourselves to a one parameter subfamily of polynomials in $F_d$. The reason for such a restriction will become clear when we get into the details of the estimates we want to establish. The corresponfing curve is the following open subset of $\Aa^1$:
 $$
 U_g=\{c\in \Aa^1(\overline{\F_q})\mid (c-t)g(t)\in F_d\}
 $$
where $g$ is a fixed element of $F_{d-1}$. The geometric points of $U_g$ can easily be seen as the complementary in $\Aa^1(\overline{\F_q})$ of the union of the set of roots of $g$ and $m$. Finally, for $c\in U_g(\F_q)$ let us denote by $E_c$ the quadratic twist of $E$ by the $\F_q$-polynomial $(c-t)g(t)\in F_d(\F_q)$. We can state a sample of our main result as follows:

  \begin{theorem} \label{estim_irred}
   With notation as above, let $L_{{\rm red},c}$ denote the reduced $L$-function of the quadratic twist $E_c$ of $E$ . For $N\geqslant 5$, $d=\deg(g)+1\geqslant d_0(E)$, and $q\geqslant q_0(E)$, we have 
    $$
    |\{c\in \A^1(\F_q)\mid g(c)\not=0, m(c)\not=0 \text{ and }  L_{{\rm red},c}\text{ is reducible }\}|\ll q^{1-\gamma}\log q\,,
    $$
  where the implied constant depends only on $E$ and where we can choose $2\gamma^{-1}=7N^2-7N+4$.
   \end{theorem}

\section{A large sieve inequality} \label{ls}

 In this section, we describe the large sieve framework thanks to which we will finally obtain quantitative estimates such as the one stated in Theorem~\ref{quant-katz}. The idea of using sieve methods in algebro-geometric contexts was first introduced by Kowalski in~\cite{KoCrelle} to study the irreducibility and the maximality of the Galois group of numerators of zeta functions of certain curves over finite fields (see~\cite[Chap. 8]{KoLS} for a general discussion on the subject and~\cite{KoIMRN} for an application to algebraic independence of the zeros of the numerators of those zeta functions).

 \subsection{The geometric setting} \label{geom-setup}
 To begin with, let us fix a parameter variety $U/\F_q$. We assume it to be smooth, affine and geometrically connected with dimension $d\geqslant 1$. If $\bar{U}$ denotes the extension of scalars $U\times \overline{\F_q}$ to a separable closure of $\F_q$ and $\bar{\eta}$ is a fixed geometric point on $U$, we may denote by $\pi_1(U,\bar{\eta})$ (resp. $\pi_1(\bar{U},\bar{\eta})$) the \'etale arithmetic (resp. geometric) fundamental group of $U/\F_q$. These groups fit the exact sequence
 \begin{equation} \label{ex-seq-pi1}
\begin{CD}
 1 @>>> \pi_1(\bar{U},\bar{\eta}) @>>>\pi_1(U,\bar{\eta}) @>{\deg}>> \hat{\Z} @>>> 1\,,
 \end{CD}
 \end{equation}
 where $\hat{\Z}$ stands for the profinite completion of $\Z$ and the degree map $\deg$ is such that $\deg(\Fr_u)=-n$ if $\Fr_u$ denotes the local geometric Frobenius at $u\in U(\F_{q^n})$, for any $n\geqslant 1$. 
 
 \par
 \medskip
  The next piece of data is, for $\ell$ running over a fixed set of primes $\Lambda$ (that will be specified in the application we have in mind), a family of sheaves $(\mt{F}_\ell)_\ell$ of $\F_\lambda$-vector spaces on $U$ (where $\F_\lambda$ is, for each $\ell$, a fixed finite extension of $\F_\ell$) with rank $N\geqslant 1$. Equivalently, for each $\ell\in\Lambda$, we are given a continuous representation
  $$
  \rho_\ell: \pi_1(U,\bar{\eta})\ra GL(N,\F_\lambda)\,.
  $$
 
 The \emph{arithmetic $\F_\ell$-monodromy group} $G_\ell$ associated to that representation is defined as its image whereas the \emph{geometric $\F_\ell$-monodromy group} attached to $\rho_\ell$ is the image $G_\ell^g$ of the subgroup $\pi_1(\bar{U},\bar{\eta})$. Unlike the case treated in~\cite{KoCrelle} where the monodromy groups involved are symplectic groups, we will have to deal with orthogonal monodromy groups in our applications. So we suppose that, for each $\ell$, we have an inclusion $G_\ell\subset O(N,\F_\ell)$ into the orthogonal group for some $\F_\ell$-quadratic space of dimension $N$.
 
 \par
 \medskip
  Finally we put the emphasis on the first major difference with the case of a geometric setting involving symplectic monodromy: we need to suppose we are given an \'etale Galois cover
  $$
  \kappa: V\ra U\,,
  $$ 
with abelian Galois group $G(V/U)$ which we assume to be an elementary $2$-group. If $\bar{\mu}$ denotes a geometric point on $V$ such that $\kappa\circ\bar{\mu}=\bar{\eta}$, then, by functoriality, we know these groups fit the exact sequence
$$
1\ra \pi_1(V,\bar{\mu})\stackrel{\pi_1\kappa}{\ra} \pi_1(U,\bar{\eta})\stackrel{\tilde{\varphi}}{\ra}G(V/U)\ra 1\,,
$$
where the quotient morphism $\phit$ is induced by the map
$$
\varphi: U(\F_q)\ra G(V/U)
$$
such that the image of $x\in U(\F_q)$ ``is'' the action of $\pi_1(U,\bar{\eta})$ on $\kappa^{-1}\{x\}$.
\par
 For each $\ell\in\Lambda$, the sheaf $\mt{F}_\ell$ can be pulled back to a lisse $\F_\ell$-adic sheaf $\kappa^*\mt{F}_\ell$ on $V$. Our crucial assumption is that the sheaf obtained has big geometric $\F_\ell$-monodromy in the sense of Hall (see~\cite{H}), for every $\ell\in\Lambda$ i.e.
 $$
 G_\ell^{g,V}=\rho_\ell(\pi_1(\bar{V},\bar{\mu}))\supset\Omega(N,\F_\ell)\,,
 $$
 where we omit to write the inclusion $\pi_1\kappa$ for brevity and where $\Omega(N,\F_\ell)$ stands for the derived group of $O(N,\F_\ell)$.
\par
Moreover we assume that the family $(\kappa^*\mt{F}_\ell)_\ell$ is \emph{linearly disjoint} in the sense that, for every $\ell\not=\ell'\in\Lambda$, the product map
 $$
\rho_{\ell,\ell'}= \rho_\ell\times\rho_{\ell'}:\pi_1(U,\bar{\eta})\ra G_\ell\times G_{\ell'}
 $$
satisfies 
$$
\rho_{\ell,\ell'}(\pi_1(\bar{V},\bar{\mu}))=G_{\ell,\ell'}^{g,V}=G_\ell^{g,V}\times G_{\ell'}^{g,V}\supset \Omega(N,\ell)\times\Omega(N,\ell')\,.
$$

 In order to unify the notation, the morphism $\rho_{\ell,\ell'}$ (resp. the group $G_{\ell,\ell'}^{g,V}$) will simply denote $\rho_\ell$ (resp. $G_\ell^{g,V}$) in the case $\ell=\ell'$.
 \par
 \medskip
 Now we obtain for each $\ell\in \Lambda$ the following diagram with exact rows and surjective downward arrows, by putting the above data together with~(\ref{ex-seq-pi1}):
 \begin{equation} \label{diagcosetsieve}
 \begin{CD} 
 1 @>>> \pi_1(\bar{V},\bar{\mu}) @>>> \pi_1(U,\bar{\eta}) @>(\deg,\phit)>> \Gamma @>>> 1\\
@. @VV \rho_\ell V @VV \rho_\ell V @VV{\rm pr}_\ell V\\
1 @>>> G_{\ell}^{g,V} @>>> G_\ell @> (\deg,\phit_\ell) >> \Gamma_\ell @>>> 1\,,
\end{CD}
 \end{equation}
where ${\rm pr}_\ell$ is defined in such a way that the diagram commutes, and, in the bottom row, $\deg$ stands for the quotient morphism $G_\ell\ra G_\ell/G_\ell^g$, with $G_\ell^g=\rho_\ell(\pi_1(\bar{U},\bar{\eta}))$.

\par
\medskip
 A few comments are in order here: first,~(\ref{diagcosetsieve}) is the analogue of the diagram $(2.2)$ in~\cite{KoCrelle}. Next, the quotient group $\Gamma\leqslant \hat{\Z}\times G(V/U)$ is abelian and the quotient group $\Gamma_\ell$ is a subgroup of the Kleinian group $\Z/2\Z\times \Z/2\Z$. Indeed the subgroup $\Omega(N,\F_\ell)$ is normal with index four in $O(N,\F_\ell)$ and the associated quotient is $\Z/2\Z\times\Z/2\Z$ (see~\cite[Section 6]{Ka} or~\cite[Section 2]{J}). The fact that, contrary to~\cite{KoCrelle}, we need to introduce the auxiliary variety $V$ is because of the assumption of linear disjointness. Indeed, as noticed in loc. cit., there are proper subgroups of $O(N,\F_\ell)\times O(N,\F_{\ell'})$ which project surjectively on both factors, e.g.
 $$
 H_{\ell,\ell'}=\{(g,g')\mid \det(g)=\det(g')\}\,,
 $$
where the equality of determinants is seen in $\{\pm 1\}$. As a consequence, if ever, in the applications, we come up with a family of sheaves $(\mt{F}_\ell)_\ell$ on $U$ with geometric monodromy as big as $O(N,\F_\ell)$, then we cannot hope for linear disjointness for $(\mt{F}_\ell)_\ell$ itself. We emphasize here that such precautions need not be taken in the case (studied in~\cite{KoCrelle}) of a family of sheaves with symplectic $\F_\ell$-monodromy.

\subsection{Sieving for Frobenius}
 
  Thanks to the data contained in~(\ref{diagcosetsieve}), we can perfrom a ``coset sieve'' as described in a very general context in~\cite[Chap. 3]{KoLS} with complements that can be found in~\cite{J}. In our case the finite set which is to be sifted is given, for a fixed $\alpha\in G(V/U)$, by
  $$
  X_\alpha=\{u\in U(\F_q)\mid \varphi(u)=\phit(\Fr_u)=\alpha\}\,.
  $$

 Note, that a particular feature of our study is that two left cosets  are considered simultaneously: on the one hand we restrict to local Frobenii with degree $-1$ (i.e. to $\F_q$-rational points on $U$) and on the other hand, among the remaining Frobenii, we only keep those mapping to $\alpha$ under $\phit$.
 \par
 Then we define $Y$ to be ``the set $X_\alpha$ seen on the fundamental group level'':
 $$
 Y=\{g^\sharp\in\pi_1(U,\bar{\eta})^\sharp\mid \phit(g^\sharp)=\alpha\ \text{ and } \deg(g^\sharp)=-1\}\,,
 $$
where $g^\sharp$ denotes the conjugacy class of $g$.

\par
\medskip
 For each $\ell\in\Lambda$, we can then define $Y_\ell=\rho_\ell(Y)$ so that if we set $\alpha_\ell={\rm pr}_\ell(-1,\alpha)$, then $Y_\ell$ is the left coset for which $\alpha_\ell$ is a representative. Now let $L\geqslant 1$ be a fixed integer and let $\Theta=(\Theta_\ell)_\ell$ be a family of conjugacy invariant subsets $\Theta_\ell\subset Y_\ell$. In this section, our purpose is to give an upper bound for the cardinality of the set
 $$
 S(X_\alpha,\Theta,L)=\{u\in X_\alpha\mid \rho_\ell(\Fr_u)\not\in \Theta_\ell \text{ for all } \ell\in \Lambda, \ell\leqslant L\}\,.
 $$
 
 Note that, alternatively, we could define $\Lambda$ as a subset of the prime numbers smaller than a fixed $L\geqslant 1$. This is in fact what we will do in the proof of Theorem~\ref{estim_petitgalois}. 
 \par
 Let $\Pi_\ell$ denote a set of representatives (containing $1$) for the equivalence classes of irreducible representations of $G_\ell$ identifying two irreducible representations $\pi$ and $\tau$ if and only if their restrictions to $G_\ell^{g,V}$ coincide. If $\Pi_\ell^*$ denotes the subset of $\Pi_\ell\setminus\{1\}$ consisting of representations $\pi$ whose character restricted to $Y_\ell$ does not identically vanish, then we have the following inequality (a proof of which can be found, in a more general context, in~\cite[Chap. 3]{KoLS})
 \begin{equation} \label{ls-ineg}
 |S(X_\alpha,\Theta,L)|\leqslant \Delta H^{-1}\,,
 \end{equation}
 where 
 $$
 H=\sum_{\ell\leqslant L}{\bigl(|\Theta_\ell|(|G_\ell^{g,V}|-|\Theta_\ell|)^{-1}\bigr)}\,,
 $$
 and
 $$
 \Delta\leqslant \max_{\ell\leqslant L}\max_{\ell\in\Pi_\ell^*}\sum_{\ell'\leqslant L}{\sum_{\tau\in\Pi_{\ell'}^*}{|W(\pi,\tau)|}}\,,
 $$
 where $W(\pi,\tau)$ is an ``exponential sum'' given by
 \begin{equation}  \label{W-expsum}
 W(\pi,\tau)=\frac{1}{\sqrt{|\hat{\Gamma}_\ell^\pi||\hat{\Gamma}_{\ell'}^\tau}|}\sum_{u\in X_\alpha}{\Tr\big(\pi\rho_\ell(\Fr_u)\big)\overline{\Tr\big(\tau\rho_{\ell'}(\Fr_u)\big)}}\,,
 \end{equation}
 and $\hat{\Gamma}_\ell^\pi=\{\psi\text{ character of } G_\ell/G_\ell^{g,V}\mid \pi\simeq\pi\otimes \psi\}$.
 
 \par
 \medskip

  The inequality~(\ref{ls-ineg}) will be refered to as the large sieve inequality and $\Delta$ will be called the large sieve constant. Obviously, obtaining an efficient upper bound for $|S(X_\alpha,\Theta,L)|$ proceeds in two steps: first we need to find a good upper bound for $\Delta$ (this is the task this section is devoted to) and second we need to produce a lower bound for $H$ (which is one of the goals of Section~\ref{main-th}). As far as $\Delta$ is concerned, the above a priori bound involving the sums $W(\pi,\tau)$ suggests that we could try to estimate those sums individually and hope for as much cancellation as possible. In order to apply that strategy we need to make a few additional assumptions. If our base variety is a curve, that involves the notion of \emph{compatible system of sheaves}:

  \begin{definition} \label{compatible-sys}
  With the same notation as above, a system of continuous representations $(\rho_\ell)_\ell: \pi_1(U,\bar{\eta})\ra GL(r,k_\ell)$ (where $k_\ell$ is a finite extension of $\F_\ell$) indexed by a set of primes $\Lambda$ not containing $p$, is said to be \emph{compatible} if there exists a number field $\mt{K}/\Q$ and for each $\ell\in\Lambda$ a prime ideal $\lambda\in\Z_{\mt{K}}$ (the ring of integers of $\mt{K}$) such that, on the one hand $\Z_{\mt{K}}/\lambda\simeq k_\ell$, and on the other hand there exists a continuous morphism
  $$
  \tilde{\rho}_\ell: \pi_1(U,\bar{\eta})\ra GL(r,\Z_{\lambda})\,,
  $$
 (where $\Z_\lambda$ is the ring of integers of the completion $\mt{K}_\lambda$ of $\mt{K}$ with respect to $\lambda$) such that
 \begin{itemize}
 \item for each $\ell\in\Lambda$, the reduction $\tilde{\rho}_\ell$ modulo $\lambda$ is isomorphic to $\rho_\ell$,
 \item for each $\ell\in\Lambda$, each extension $\F_{q^n}/\F_q$ and each $u\in U(\F_{q^n})$, the reversed characteristic polynomial
 $$
 \det(1-T\tilde{\rho}_\ell(\Fr_u))\in \Z_\lambda[T]
 $$
 has coefficients in $\Z_\mt{K}$ and is independent of $\ell$. 
 \end{itemize} 
  \end{definition}

 We can now state the key proposition, giving the expected cancellation for the sums $W(\pi,\tau)$ under certain condtions.

 \begin{proposition}\label{estim-W}
  With the same notation as above, suppose $\pi\in\Pi_\ell^*$ and $\tau\in\Pi_{\ell'}^*$. Then the normalised exponential sums $W(\pi,\tau)$ satisfy:
  \begin{enumerate}
  \item if $G_\ell^{g,V}$ has order prime to $p$ for all $\ell\in\Lambda$, then
  $$
  W(\pi,\tau)=\delta((\ell,\pi),(\ell',\tau))\eps_\pi q^d+O\Bigl(q^{d-1/2}|G_{\ell,\ell'}||G(V/U)|(\dim\pi)(\dim\tau)\Bigr)
  $$
  where the implied constant depends only on $\bar{U}$.
  \item if $U$ is a curve (i.e. $d=1$) and $(\mt{F}_\ell)_\ell$ forms a compatible system of sheaves (obtained as the reduction of a system $(\tilde{\mt{F}}_\ell)_\ell$ of $\Z_\ell$-adic sheaves), then
  $$
  W(\pi,\tau)=\delta((\ell,\pi),(\ell',\tau))\eps_\pi q+O\Bigl(q^{1/2}(\dim\pi)(\dim\tau)\Bigr)\,,
  $$
  where the implied constant only depends on the compactly supported Euler-Poincar\'e characteristic of $\bar{U}$ and on the system $(\tilde{\mt{F}}_\ell)$ on $\bar{U}$.
  \par
  \medskip
  Moreover, in both cases, the constant $\eps_\pi$ is uniformly bounded by $1$ (i.e. $\eps_\pi\leqslant 1$ independently of $\ell$ and $\pi$).
  \end{enumerate} 
 \end{proposition}

 \begin{proof}
  Forgetting for a while the normalisation factor in~(\ref{W-expsum}), we consider
  $$
  W'(\pi,\tau)=\sum_{u\in X_\alpha}{\Tr\big(\pi\rho_\ell(\Fr_u)\big)\overline{\Tr\big(\tau\rho_{\ell'}(\Fr_u)\big)}}\,.
  $$
  
  Using Frobenius reiprocity yields
  $$
  W'(\pi,\tau)=\sum_{u\in U(\F_q)}{\Tr\big(\pi\rho_\ell(\Fr_u)\big)\overline{\Tr\big(\tau\rho_{\ell'}(\Fr_u)\big)}}\frac{1}{|G(V/U)|}\sum_{\chi\in\hat{G}(V/U)}{\overline{\chi(\alpha)}\chi(\phit(\Fr_u))}\,,
  $$
  so that we are naturally led to consider the sums
  $$
  W_\chi(\pi,\tau)=\sum_{u\in U(\F_q)}{\Tr\big(\pi\rho_\ell(\Fr_u)\big)\overline{\Tr\big(\tau\rho_{\ell'}(\Fr_u)\big)\chi(\phit(\Fr_u))}}\,,
  $$
  for each character $\chi\in\hat{G}(V/U)$. The last two sums are thus related by the formula
  $$
   W'(\pi,\tau)=\frac{1}{|G(V/U)|}\sum_{\chi\in\hat{G}(V/U)}{\overline{\chi(\alpha)}W_\chi(\pi,\tau)}\,.
  $$
  
  For a fixed $\chi$, the sum $W_\chi(\pi,\tau)$ is closely related to the representation $\pi\rho_\ell\otimes\overline{\tau\rho_{\ell'}}\otimes\chi\phit$ of the arithmetic fundamental group $\pi_1(U,\bar{\eta})$. To that representation corresponds an \'etale sheaf $\mt{F}_\chi(\pi,\tau)$, so that, applying the Grothendieck-Lefschetz trace formula (see~\cite{Gr}), we get
  $$
  W_\chi(\pi,\tau)=\sum_{i=0}^{2d}{(-1)^i\Tr(\Fr\mid H_c^i(\bar{U},\mt{F}_\chi(\pi,\tau)))}\,,
  $$
 where $\Fr$ denotes the morphism induced on the compactly supported \'etale cohomology by the global (geometric) Frobenius on $\bar{U}$.
 
 \par
 \medskip
  The whole point of the proof is to give a precise analysis of that sum. First let us denote
  $$
  [\pi,\bar{\tau}]=\begin{cases}&\pi\boxtimes\bar{\tau}\text{ if }\ell\not=\ell'\,,\\
                                &\pi\otimes\bar{\tau}\text{ if } \ell=\ell'\,,   
  \end{cases}
  $$
  where we distinguish the ``external'' tensor product ``$\boxtimes$'' from the ``internal'' tensor product ``$\otimes$''. Then, using the same notation as above, we have
  $$
    \pi\rho_\ell\otimes\overline{\tau\rho_{\ell'}}=[\pi,\bar{\tau}]\rho_{\ell,\ell'}\,.
  $$
  
  Using the additional notation $G_{\ell,\ell'}=\rho_{\ell,\ell'}(\pi_1(U,\bar{\eta}))$, we see that the representation $[\pi,\bar{\tau}]\rho_{\ell,\ell'}\otimes\chi\tilde{\varphi}$ corresponding to $\mt{F}_\chi(\pi,\tau)$, factors through a finite group:
  $$
  \pi_1(U,\bar{\eta})\stackrel{\rho_{\ell,\ell'}\times\tilde{\varphi}}{\ra} G_{\ell,\ell'}\times G(V/U)\stackrel{[\pi,\bar{\tau}]\boxtimes \chi}{\ra} GL((\dim \pi)(\dim \tau),\C)\,.
  $$
  
  So any eigenvalue of any $\Fr_u$, with $u\in U(\F_{q^n})$ acting on the fiber of $\mt{F}_\chi(\pi,\tau)$ at a geometric point $\bar{u}$ lying above $u$ is a root of unity. In other words, $\mt{F}_\chi(\pi,\tau)$ is pointwise pure of weight $0$. From a celebrated Theorem of Deligne~\cite[page 138]{deligne_weilII}, we deduce that the global geometric Frobenius $\Fr$ acts on each $H^i_c(\bar{U},\mt{F}_\chi(\pi,\tau))$ with eigenvalues of modulus $\leqslant q^w$ for some integer $w\leqslant i/2$.
 
 \par
 \medskip
  Now in the sum $W_\chi(\pi,\tau)$, the main contribution comes from the trace of Frobenius on the cohomology space with highest degree $H_c^{2d}(\bar{U},\mt{F}_\chi(\pi,\tau))$. The above argument yields the following estimate for the sum of the other contributions:

  \begin{lemme} \label{somme_betti}
  %Let $\sigma_c'(\bar{U},\mt{F}_\chi(\pi,\tau))=\sum_{i=0}^{2d-1}{\dim H_c^i(\bar{U},\mt{F}_\chi(\pi,\tau)))}$.
   We have the following estimates:
  \begin{enumerate}
  \item if $G_{\ell,\ell'}$ has order prime to $p$, then
  $$
  \bigl|\sum_{i=0}^{2d-1}{(-1)^i\Tr(\Fr\mid H_c^i(\bar{U},\mt{F}_\chi(\pi,\tau)))}\bigr|\leqslant q^{d-1/2}C(\bar{U})|G_{\ell,\ell'}||G(V/U)|(\dim \pi)(\dim \tau)\,,
  $$
  where $C(\bar{U})$ is a constant depending only on $\bar{U}$.
  \item if $d=1$ (i.e. $U$ is a curve) and the family $(\rho_\ell)$ is obtained by reducing a compatible system of $\ell$-adic sheaves, then
  $$
  \bigl|\sum_{i=0}^{2d-1}{(-1)^i\Tr(\Fr\mid H_c^i(\bar{U},\mt{F}_\chi(\pi,\tau)))}\bigr|\leqslant q^{1/2}C(\bar{U},(\rho_\ell))(\dim \pi)(\dim \tau)\,,
  $$
  where the constant $C(\bar{U},(\rho_\ell))$ depends only on the compactly supported Euler-Poincar\'e characteristic of $\bar{U}$ and the compatible system.
  \end{enumerate}
  \end{lemme}

  \begin{proof}
  This follows directly from the discussion preceding the statement of the lemma and~\cite[Prop. 8.9]{KoLS} (to which we refer for precise expressions for the constants $C(\bar{U})$
   in case $(1)$ and $C(\bar{U},(\rho_\ell))$ in case $(2)$).
   \par
   For instance, to derive $(2)$ from loc. cit., we need only notice that $(2)$ of~\cite[Prop. 8.9]{KoLS} holds for a representation $\rho$ of the type $\rho_{\ell,\ell'}\times\tilde{\varphi}$ with values in $G_{\ell,\ell'}\times G(V/U)$ since $G(V/U)$ is assumed to be an elementary $2$-group and thus can be seen as a subgroup of a product of type $\prod_i GL(1,k_i)$ where $k_i$ is a finite field with characteristic $\ell_i\not=p$. 
  \end{proof}

   It remains now to determine precisely what the main contribution is. It is enough to determine the dimension of the cohomology space of degree $2d$ and then to show that the global geometric Frobenius acts with eigenvalues which are all equal and have for common value $\pm q^d$. Notice first that we have, from the Hochschild-Serre spectral sequence,
  $$
 H_c^{2d}(\bar{U},\mt{F}_\chi(\pi,\tau))=H_c^{2d}(\bar{V},\kappa^*\mt{F}_\chi(\pi,\tau))^{G^g(V/U)}\,, 
  $$
  where the group $G^g(V/U)$ is the ``geometric'' version of the Galois group of the cover $\kappa :V\ra U$, i.e. it is defined via the exact sequence
  $$
  1\ra \pi_1(\bar{V},\bar{\mu})\ra\pi_1(\bar{U},\bar{\eta})\ra G^g(V/U)\ra 1\,.
  $$
  
   The cohomology space $H_c^{2d}(\bar{V},\kappa^*\mt{F}_\chi(\pi,\tau))$ coincides with the space of coinvariants (see~\cite[Sommes trig., Rem. 1.18d]{De})
   $$   
   \bigl(\pi\rho_\ell\otimes\overline{\tau\rho_{\ell'}}\otimes\chi\phit\bigr)_{\pi_1(\bar{V},\bar{\mu})}(-d)\,,
   $$
 which is the $d$-th Tate twist of the space of coinvariants of the representation $\pi\rho_\ell\otimes\overline{\tau\rho_{\ell'}}\otimes\chi\phit$ seen as a $\pi_1(\bar{V},\bar{\mu})$-module. First, we focus on dimension issues. In the last part of the proof we will show that Frobenius acts as $\pm {\rm Id}$ on the space of coinvariants $H_c^{2d}(\bar{U},\mt{F}_\chi(\pi,\tau))$ (or equivalently with eigenvalues all equal to $q^d$ or all equal to $-q^d$ on the $d$-th Tate twist of that space).

\par\medskip
 The Tate twist having no impact on the dimension, we end up trying to evaluate the dimension of the subspace of invariants:
 $$ \Bigl(\bigl(\pi\rho_\ell\otimes\overline{\tau\rho_{\ell'}}\otimes\chi\phit\bigr)_{\pi_1(\bar{V},\bar{\mu})}\Bigr)^{G^g(V/U)}\,.
 $$
 
 As $\chi\phit$ is trivial as a $\pi_1(\bar{V},\bar{\eta})$-module, we have an isomorphism
 $$
\bigl(\pi\rho_\ell\otimes\overline{\tau\rho_{\ell'}}\otimes\chi\phit\bigr)_{\pi_1(\bar{V},\bar{\mu})}\simeq \bigl(\pi\rho_\ell\otimes\overline{\tau\rho_{\ell'}}\bigr)_{\pi_1(\bar{V},\bar{\mu})}\otimes\chi\phit\,.
 $$ 
 
  Thanks to the assumption of linear disjointness, we have the isomorphism between the spaces of coinvariants
  $$
  \bigl(\pi\rho_\ell\otimes\overline{\tau\rho_{\ell'}}\bigr)_{\pi_1(\bar{V},\bar{\mu})}=([\pi,\bar{\tau}]\rho_{\ell,\ell'})_{\pi_1(\bar{V},\bar{\mu})}\simeq [\pi,\bar{\tau}]_{G_{\ell,\ell'}^{g,V}}
  $$

  So, as we are now dealing with complex representations of finite groups, that last space of coinvariants coincides with the space of invariants $[\pi,\bar{\tau}]^{G_{\ell,\ell'}^{g,V}}$. As $\pi\in\Pi_\ell^*$ and $\tau\in\Pi_{\ell'}^*$, we can invoke~\cite[Lemma 3.4]{KoLS} to deduce that this space of invariants can only be non trivial in the case where $(\ell,\pi)=(\ell',\tau)$. As a consequence the dimension we are interested in is that of the subspace of the $G^g(V/U)$-invariants of
  $$ 
  \bigl(\pi\rho_\ell\otimes\overline{\pi\rho_{\ell}}\bigr)_{\pi_1(\bar{V},\bar{\mu})}\otimes\chi\phit\,.
  $$
 
  Applying once more the above arguments, we deduce that the dimension we actually need to compute is that of the space (recall that $\chi\phit$ is a trivial $\pi_1(\bar{V},\bar{\mu})$-module)
  $$ \Bigl(\bigl(\pi\rho_\ell\otimes\overline{\pi\rho_{\ell}}\otimes\chi\phit\bigr)^{\pi_1(\bar{V},\bar{\mu})}\Bigr)^{G^g(V/U)}\,.
  $$
  
   To that purpose, we use the following lemma, which is purely representation theoretic.

   \begin{lemme} \label{lem-rep-th}
   Let $G$ be a compact group, $\Omega$ a normal compact subgroup of $G$ wtih abelian quotient $K$ fitting the exact sequence
   $$
   1\ra \Omega\ra G\ra K\ra 1\,.
   $$ 
   
   Let $k$ be an algebraically closed field with characteristic zero, $\pi$ a finite dimensional $k$-representation of $G$ and $\chi$ a (degree $1$) character of $K$. If $\bar{\pi}$ denotes the contragredient representation of $\pi$, then we have the isomorphism of $G$-modules
   $$
   \bigl((\pi\otimes\bar{\pi}\otimes\chi)^\Omega\bigr)^K\simeq (\pi\otimes\bar{\pi}\otimes\chi)^{G}\,,
   $$
  % the $\bar{\chi}$-component of the representation $\pi\otimes\bar{\pi}$.
   \end{lemme}

   \begin{proof}
 The space of invariants $(\pi\otimes\bar{\pi}\otimes\chi)^\Omega$ coincides with the projection of the restriction $(\pi\otimes\bar{\pi}\otimes\chi)_{|\Omega}$ on the trivial representation of $\Omega$. Using Frobenius reciprocity, this space can also be written
 $$
 \bigoplus_{\psi\in\hat{K}}{(\pi\otimes\bar{\pi}\otimes\chi)(\psi)}\,,
 $$   
which is nothing but the sum of the $\psi$-components of $\pi\otimes\bar{\pi}\otimes\chi$ (seen as a $G$-module) over the characters $\psi$ of $K$. This space can also be written
 $$
 \bigoplus_{\psi\in\hat{K}}{(\pi\otimes\bar{\pi})(\bar{\chi}\psi)}\,.
 $$
 
  Under this form, it is easy to deduce that the subspace of $K$-invariants of that representation is $(\pi\otimes\bar{\pi})(\bar{\chi})$, obtained by projecting on the trivial representation $1_K$. Finally, this subspace can also be written $(\pi\otimes\bar{\pi}\otimes\chi)(1)$ which, by definition, is the space of invariants
  $$
  (\pi\otimes\bar{\pi}\otimes\chi)^{G}\,.
  $$
   \end{proof}

   Thanks to the lemma, we can now give a useful expression for the dimension of the space of invariants we are interested in.
   
   \begin{lemme} \label{character_lowerline}
   If $\bigl(\pi\rho_\ell\otimes\overline{\pi\rho_{\ell}}\otimes\chi\phit\bigr)^{\pi_1(\bar{U},\bar{\eta})}\not= 0$, then there exists a character $\psi_\ell^\chi$ of $G_\ell$ such that we have an isomorphism between the spaces of invariants
   $$
  \bigl([\pi,\bar{\pi}]\rho_\ell\otimes\chi\phit\bigr)^{\pi_1(\bar{U},\bar{\eta})}\simeq ([\pi,\bar{\pi}]\otimes\psi_\ell^\chi)^{G_\ell^g}\,.
   $$
   Moreover the dimension of the latter space is 
   $$
  \gamma_\ell^{\pi,\chi}=|\{\psi\text{ character of } G_\ell/G_\ell^{g}\mid \pi\simeq \pi\otimes\overline{\psi_\ell^\chi}\psi \}|\,.  
   $$
   \end{lemme}

   \begin{proof}
   Assuming the space of invariants we consider is not trivial, let us fix a nonzero $\pi_1(\bar{U},\bar{\eta})$-invariant vector $v$. We have
   $$
   [\pi,\bar{\pi}](g)v=\overline{\chi\phit}(g)v,\text{ for all }g\in \pi_1(\bar{U},\bar{\eta})\,.
   $$
      
    For any $g$, we have $\chi\phit(g)\in\{\pm 1\}$ and that quantity depends only on $\rho_\ell(g)$. Indeed, if $g_1, g_2$ are such that $\rho_\ell(g_1)=\rho_\ell(g_2)$, then the above relation yields $\chi\phit(g_1)=\chi\phit(g_2)$ since we chose $v\not= 0$. That observation enables us to define a character $\psi_\ell^\chi$ of $G_\ell^g$:
    $$
    \psi_\ell^\chi(h)=\overline{\chi\phit}(g),\text{ where } h\in G_\ell^g \text{ and } g \text{ is such that } \rho_\ell(g)=h\,.
    $$
  
  Notice that such a character $\psi_\ell^\chi$ is trivial on $G_\ell^{g,V}$ simply because $\chi\phit$ is trivial on $\pi_1(\bar{V},\bar{\mu})$. Thus, we can see it as a character of $G_\ell^g/G_\ell^{g,V}$ and it can be extended to a character of $G_\ell/G_\ell^{g,V}$ by making it act trivially on the possibly ``missing'' $\Z/2\Z$ factor (this is because we have the inclusions
  $$
  \Omega(N,\F_\ell)\subset G_\ell^{g,V}\subset G_\ell^g\subset O(N,\F_\ell)\,,
  $$
  and we know $O(N,\F_\ell)/\Omega(N,\F_\ell)\simeq \Z/2\Z/\times \Z/2\Z$).

  % we are reduced to determining the dimension of the space spanned by the vectors $v$ such that, for any $g\in\pi_1(\bar{U},\bar{\eta})$,
  % $$ (\pi\rho_\ell\otimes\overline{\pi\rho_{\ell}}\otimes\chi\phit)(g)(v)=1\,.
  % $$
   
  % Now $\chi$ is a character of order at most $2$ (the group $G(V/U)$ being of exponent $2$) and $\chi\phit$ restricts trivially to $\pi_1(\bar{V},\bar{\mu})$, so there exists a character $\psi_\ell^\chi$ of $G_\ell^g/G_\ell^{g,V}$ (which also has order at most $2$) such that, for all $g\in\pi_1(\bar{U},\bar{\eta})$, we have
  % $$
  % \chi\phit(g)=\psi_\ell^\chi\rho_\ell(g)\,.
  % $$

 %Note that the character $\psi_\ell^\chi$ can easily be extended to a character of $G_\ell/G_\ell^{g,V}$: indeed that quotient is isomorphic to a product of copies of $\Z/2\Z$ and it contains, up to isomorphism, the group $G_\ell^g/G_\ell^{g,V}$, so we can define the extension of $\psi_\ell^\chi$ as being trivial on the ``missing'' $\Z/2\Z$ factors.
 %\par
\medskip
  To prove the second part of the statement we need to determine the dimension of the space spanned by the vectors $w$ satisfying
 $$
 \pi\otimes\bar{\pi}(g)(w)=\overline{\psi_\ell^\chi(g)}w\,,
 $$
 for all $g\in G_\ell^g$. Using Frobenius reciprocity once more, we can compute that dimension (equal to the multiplicity of the trivial character in the restriction of $\pi\otimes\bar{\pi}\otimes\psi_\ell^\chi$ to the group $G_\ell^g$):
 $$
 \langle \pi\otimes\bar{\pi}\otimes\psi_\ell^\chi;1\rangle_{G_\ell^g}= \langle\pi\otimes\bar{\pi}\otimes\psi_\ell^\chi;\sum_{\psi}{\psi}\rangle_{G_\ell}\,,
 $$
 where, in the last sum, $\psi$ runs over the group of characters of the abelian group $G_\ell/G_\ell^g$.
 \par
 As we have chosen $\pi\in\Pi_\ell^*$, the scalar product in the right hand side of the above equality is equal to the cardinality
 $$
 \gamma_\ell^{\pi,\chi}=|\{\psi\text{ character of } G_\ell/G_\ell^g\mid \pi\simeq\pi\otimes\overline{\psi_\ell^\chi}\psi\}|\,.
 $$
   \end{proof}
 
 \begin{remark}
 We clearly have an injection
 \begin{align} \label{injection}
 \{\psi\text{ character of } G_\ell/G_\ell^g\mid \pi\simeq\pi\otimes\overline{\psi_\ell^\chi}\psi\}&\ra \hat{\Gamma}_\ell^\pi\\ \nonumber
 \psi &\mapsto\overline{\psi_\ell^\chi}\psi\,,
 \end{align}
so that $\gamma_\ell^{\pi,\chi}\leqslant |\hat{\Gamma}_\ell^\pi|$ for any choice of $\pi\in\Pi_\ell^*$ and $\chi\in \widehat{G(V/U)}$.
 \end{remark}  
   
\par
\medskip

 To unify the notation, we set $\gamma_\ell^{\pi,\chi}=0$ if $([\pi,\bar{\pi}]\rho_\ell\otimes\chi\phit)^{\pi_1(\bar{U},\bar{\eta})}$ is trivial. The following lemma enables us to determine completely the contribution of the trace of Frobenius on $H^{2d}_c(\bar{U},\mt{F}_\chi(\pi,\pi))$.

\begin{lemme}
Provided $H^{2d}_c(\bar{U},\mt{F}_\chi(\pi,\pi))$ is non zero, there is a unique eigenvalue for the global geometric Frobenius acting on that space: it is either $q^d$ or $-q^d$.
\end{lemme}

\begin{proof}
 It is enough to prove that the global geometric Frobenius acts on the space of coinvariants $(\pi\rho_\ell\otimes\bar{\pi}\rho_\ell\otimes\chi)_{\pi_1(\bar{U},\bar{\eta})}$ as either ${\rm Id}$ or $-{\rm Id}$ (so it acts on the $d$-th Tate twist by scalar multiplication by $q^d$ or $-q^{d}$). From the arguments preceding Lemma~\ref{lem-rep-th} and Lemma~\ref{lem-rep-th} itself  we know that the space of coinvariants we are interested in is isomorphic (as a Galois module) to the space of invariants
$$
(\pi\rho_\ell\otimes\bar{\pi}\rho_\ell\otimes\chi\tilde{\phi})^{\pi_1(\bar{U},\bar{\eta})}\,.
$$

That space of invariants can be seen as a $\hat{\Z}$-module for which we want to understand the action of its topological generator $d(\Fr_u)=-1$ for any $u\in U(F_q)$ such that $\phit(\Fr_u)=\alpha$. Now from Lemma~\ref{character_lowerline} that action is the same as that of ${\rm pr}_\ell(-1,\alpha)$ on
$([\pi,\bar{\pi}]\otimes\psi_\ell^\chi)^{G_\ell^g}$. %provided the space of invariants we start with is not trivial. 
\par
 We saw in the course of the proof of Lemma~\ref{character_lowerline}, that $([\pi,\bar{\pi}]\otimes\psi_\ell^\chi)^{G_\ell^g}$ decomposes as 
 $$
 \bigoplus\{\psi\text{ character of } G_\ell/G_\ell^{g}\mid \pi\simeq\pi\otimes\overline{\psi_\ell}\psi\}\,.
 $$
 
  On each of these $1$-dimensional summands ${\rm pr}_\ell(-1,\alpha)$ acts by multiplication by $\psi({\rm pr}_\ell(-1),\alpha)$. For each $x\in Y_\ell$ and each $\psi$ involved in the above sum, we have $\psi(x)=\psi({\rm pr}_\ell(-1,\alpha))$. Now taking the trace on each side of the isomorphism defining the characters $\psi$ we focus on, we get for all $x\in Y_\ell$:
  $$
  \Tr \pi(x)=\Tr\pi(x)\overline{\psi_\ell^\chi}(x)\psi(x)\,.
  $$
  
  As we chose $\pi\in\Pi_\ell^*$, the function $\Tr\pi$ does not identically vanish on $Y_\ell$ and we get $\psi({\rm pr}_\ell(-1,\alpha))=\psi_\ell^\chi({\rm pr}_\ell(-1,\alpha))$ for all $\psi$. The result follows from the fact that $\psi_\ell^\chi$ takes its values in $\{\pm 1\}$.
  
\end{proof}

 Emphasizing the main term in the sum $W_\chi(\pi,\tau)$, we have just proven that
 $$
 W_\chi(\pi,\tau)=\delta((\ell,\pi),(\ell',\tau))\gamma_\ell^{\pi,\chi}q^d+\sum_{i=0}^{2d-1}{(-1)^i\Tr\bigl(\Fr\mid H_c^i(\bar{U},\mt{F}_\chi(\pi,\tau))\bigr)}\,.
 $$
 Now recall
 $$
 W(\pi,\tau)=\bigl(|\hat{\Gamma}_\ell^\pi||\hat{\Gamma}_{\ell'}^\tau|\bigr)^{-1/2}|G(V/U)|^{-1}\sum_{\chi\in \widehat{G(V/U)}}{\overline{\chi(\alpha)}W_\chi(\pi,\tau)}\,.
 $$
 Setting
 $$
 \eps_\pi=\frac{1}{|G(V/U)|}\sum_\chi{\bigl(\overline{\chi(\alpha)}(\gamma_\ell^{\pi,\chi}|\hat{\Gamma}_\ell^\pi|^{-1})\bigr)}\,,
 $$
 and looking back at~(\ref{injection}), we get the main term as stated in the proposition.
 
 Applying Lemma~\ref{somme_betti} we finally get the full statement of the proposition.

 %\par \medskip
 %For the proof of Proposition~\ref{estim-W} to be complete, we need to give a suitable upper bound for the sum of the right hand side of the above equality. To that end we invoke Deligne's Theorem once more to produce the upper bound:
 %$$
 %\sum_{i=0}^{2d-1}{(-1)^i\Tr\bigl(\Fr\mid H_c^i(\bar{U},\mt{F}_\chi(\pi,\tau))\bigr)}=O(q^{d-1/2}\sigma_c(\bar{U},\mt{F}_\chi(\pi,\tau))\,,
 %$$
 %where $\sigma_c(\bar{U},\mt{F}_\chi(\pi,\tau))$ is the sum over $i\in\{0,\ldots,2d\}$ of the Betti numbers $\dim H_c^i(\bar{U},\mt{F}_\chi(\pi,\tau))$, and the implied constant is absolute.
 
 %\par
 %Now, to obtain the upper bound we need for $\sigma_c(\bar{U},\mt{F}_\chi(\pi,\tau))$, we invoke~\cite[Prop. 8.9]{KoLS}. It is indeed straightforward to verify that the proof of that result still holds if the sheaf $\pi(\rho)$ of loc. cit. is replaced by the sheaf $\pi(\rho)\otimes\chi$ where $\chi$ is a character of order $2$ of $\pi_1(U,\bar{\eta})$ and if we keep the same assumptions for $\pi$ and $\rho$. This completes the proof of Proposition~\ref{estim-W}.

 \end{proof}

 Thanks to Proposition~\ref{estim-W}, we can deduce an upper bound for the large sieve constant $\Delta$ that in turn yields the following result:
 
 \begin{corollaire} \label{up-bound-delta}
 With the same assumptions as in Proposition~\ref{estim-W}, let us denote $d'=N(N-1)/2$ the dimension of ${\bf O}(N)$ as an algebraic group, then we have
 $$
 |\{u\in U(\F_q)\mid \rho_\ell(\Fr_u)\not\in\Theta_\ell\text{ for all } \ell\in \Lambda\}|\leqslant |G(V/U)|(q^d+Cq^{d-1/2}(L+1)^A)H^{-1}\,,
 $$
 where, 
 \begin{itemize}
 \item In case $(1)$ of Proposition~\ref{estim-W}, $A=7d'/2+1$ and $C$ depends only on $\bar{U}$.
 \item In case $(2)$ of Proposition~\ref{estim-W}, $A=3d'/2+1$ and $C$ depends only on the Euler-Poincar\'e  characteristic of $\bar{U}$ and on the system $(\tilde{\mt{F}}_\ell)$ on $\bar{U}$.
 \end{itemize}
 \end{corollaire}

\begin{proof}
 Let us fix $\alpha\in G(V/U)$. Combining~(\ref{ls-ineg}) and Proposition~\ref{estim-W}, we get, in case $(1)$ of Proposition~\ref{estim-W},
 $$
 |S(X_\alpha,\Theta;L)|\leqslant \max_{\ell,\pi}\Bigl\{q^d+Cq^{d-1/2}(\dim \pi)\sum_{\ell',\tau}{(\dim\tau)|G_{\ell,\ell'}|}\Bigr\}H^{-1}\,,
 $$
 whereas, in case $(2)$ of Proposition~\ref{estim-W},
 $$
 |S(X_\alpha,\Theta;L)|\leqslant \max_{\ell,\pi}\Bigl\{q^d+Cq^{d-1/2}(\dim \pi)\sum_{\ell',\tau}{(\dim\tau)}\Bigr\}H^{-1}\,.
 $$
 
  Now we can use the following trivial inequalities
  $$
  \dim\pi\leqslant\sqrt{|G_\ell|}\,,\indent \sum_{\tau\, { \rm irr}}{\dim\tau}\leqslant |G_{\ell'}|\,,\indent |G_{\ell,\ell'}|\leqslant |G_\ell||G_{\ell'}|\,,
  $$
  where $\pi$ is any irreducible representation of $G_\ell$ and $\tau$ runs over a set of representatives for the isomorphism classes of irreducible representations of $G_{\ell'}$. Doing so, we eventually get the upper bound:
  $$
   |S(X_\alpha,\Theta;L)|\leqslant (q^d+Cq^{d-1/2}(L+1)^A)H^{-1}\,,
  $$
  since we have
  $$
  |G_\ell|\leqslant |O(N,\F_\ell)|\leqslant (\ell+1)^{N(N-1)/2}\,.
  $$
  
   To conclude the proof, we need only remark:
  $$
  |\{u\in U(\F_q)\mid \rho_\ell(\Fr_u)\not\in\Theta_\ell, \text{ for all } \ell\leqslant L\}|=\sum_{\alpha\in G(V/U)}{|S(X_\alpha,\Theta;L)|}\,.
  $$
  
   Thus the left hand side of the above equality can be bounded by
   $$
  |G(V/U)|\max_{\alpha\in G(V/U)}{|S(X_\alpha,\Theta;L)|}\,.
   $$
    
    Combining that with the upper bound for $|S(X_\alpha, \Theta;L)|$ we get the estimate we wanted.
\end{proof}

\begin{remark}
 It seems reasonable to hope that for any representation $\pi\in\Pi_\ell^*$ and any nontrivial $\chi$, we have $\gamma_\ell^{\pi,\chi}=0$. As we trivially see that $\gamma_\ell^{\pi,1}=|\hat{\Gamma}_\ell^\pi|$, then this would enable us to set $\eps_\pi=|G(V/U)|^{-1}$ for any $\pi$, which would be very useful both to save a power of $2$ in the implied constant of Theorem~\ref{quant-katz} and for uniformity purposes. Indeed, without that kind of additional data, we cannot get rid of the constant $|G(V/U)|$ (i.e. of the dependency on the degree of the cover $V\ra U$) in corollary~\ref{up-bound-delta}.
\end{remark}
 
  To deduce Theorem~\ref{estim_irred} from the above corollary, we need to give a lower bound for the constant $H$. As we are interested in uniformity issues (i.e. we would like the implied constant to be an absolute constant in our estimate, at least in the case where $E$ is a Legendre curve), that lower bound needs to be explicit in terms of the common degree $N$ of the $L$-functions studied.

 \section{Statement and proof of the main result}\label{main-th}

 Applying the sieving machinery of the preceding section, we can prove a more general version of Theorem~\ref{estim_irred} in the context of Section~\ref{quant-katz}. Indeed, beyond the question of the irreducibility of $L_{{\rm red},c}$ when averaging over $c$, we can investigate the question of the maximality of the Galois group of such a $\Q$-polynomial.

\subsection{A stronger version of Theorem~\ref{quant-katz}}
 First, we need to understand why the setup of Section~\ref{quant-katz} fits the abstract geometric framework of Section~\ref{geom-setup}. This is mostly due to Hall (see~\cite{H}). The parameter variety we are interested in is the affine curve $U_g$ of Section~\ref{quant-katz}, where $g$ is a fixed element of $F_{d-1}$. We restrict ourselves to the curve $U_g$ because we cannot guarantee that the condition $p\nmid |G_\ell^{g,V}|$ holds, as would be required by $(1)$ of Proposition~\ref{estim-W} if we wanted to work with the variety $F_d$.
 \par
 With notation as in Section~\ref{quant-katz}, there is, for each $\ell$ a unique lisse $\F_\ell$-sheaf denoted $\mt{T}_{d,\ell}$ whose fiber over any $f\in F_d(\F_q)$ is $\mt{V}_{f,\ell}$. That sheaf can be restricted to the curve $U_g$ and this gives rise, for a geometric point $f(t)=(c-t)g(t)\in F_d(\overline{\F_q})$ to a representation:
 $$
 \pi_1(U_g,\bar{c})\ra GL(\mt{V}_{f,\ell})\,,
 $$
 where $\bar{c}$ is a geometric point over $c$.
 
 \par
 \medskip
  From what we saw at the end of Section~\ref{quad-twist}, we know that the image of that representation sits in the orthogonal group $O(\mt{V}_{f,\ell})$ (with respect to the symmetric pairing given by Poincar\'e duality). That means both the arithmetic and the geometric monodromy groups of $\mt{T}_{d,\ell}$ are subgroups of $O(\mt{V}_{f,\ell})$. As explained by Hall in~\cite[Section 2]{H}, the system of sheaves $(\mt{T}_{d,\ell})_\ell$ restricted to $U_g$ forms a compatible system in the sense of Definition~\ref{compatible-sys}. Indeed we have
  $$
  L_{\rm red}(E_f/K;T)\equiv \det(1-T\Fr_q\mid \mt{V}_{f,\ell})\,(\mathrm{mod}\,\ell)\,,
  $$ 
 which of course is a crucial identity for us as we want to study $L_{\rm red}(E_f/K;T)$ via sieve methods.
 
 \par
 \medskip
 
  Finally we need to understand why there exists an \'etale Galois cover $V_g\ra U_g$ satisfying the properties stated in Section~\ref{geom-setup}. That is in fact a reformulation of Hall's Theorem $6.6$ in~\cite{H}. The following lemma, which was explained to the author by C. Hall, provides us with a precise version of the property we need.

\begin{lemme}\label{lemme-chris}
 With notation as above,  if we assume that $d=\deg g+1\geqslant d_0(E)$ and $\ell\geqslant \ell_0(E)$, then there exists an \'etale Galois cover $V_g\ra U_g$, with Galois group $G(V_g/U_g)$ an elementary $2$-group, such that, if we denote by $\rho_\ell$ the representation which ``is'' $\mt{T}_{d,\ell}$, we have
$$
\rho_\ell(\pi_1(\overline{V_g},\bar{\mu}))=\Omega(N,\F_\ell)\,,
$$ 
and
$$
\rho_{\ell,\ell'}(\pi_1(\overline{V_g},\bar{\mu}))=\Omega(N,\F_\ell)\times\Omega(N,\F_{\ell'})\,,
$$
for $\ell\not=\ell'$.

\par
\medskip
 Moreover we can choose
$$
\ell_0(E)=\max\bigl(13,\max\{p'\text{ prime}\mid v_{p'}(-{\rm ord}_P(j(E)))>0\text{ for some } P\in \F_q[t]\text{ irreducible}\}\bigr)\,.
$$
\end{lemme}

\begin{proof}
 Let $U_{\ell,g}\ra U_g$ be the cover with Galois group $G_\ell$. Because $\mt{T}_{d,\ell}$ (or $\rho_\ell$) has big geometric monodromy for $d$ and $\ell$ large enough (see~\cite[Th. 6.3]{H}), the group $G_\ell^g$ (which is a subgroup of $G_\ell$) contains the derived group $\Omega(N,\F_\ell)$ of the orthogonal group as soon as $\ell$ is sufficiently large. If $W_{\ell,g}\ra U_g$ is the subcover corresponding to $\Omega(N,\F_\ell)$ then $W_{\ell,g}\ra U_g$ is Galois with group a subgroup of $\Z/2\Z\times \Z/2\Z$. Thus if we define $V_g$ to be the compositum of these covers, the cover $V_g\ra U_g$ we get is still Galois with group an elementary $2$-group. Now $\Omega(N,\F_\ell)$ has no non trivial abelian quotient (see e.g.~\cite[Prop. 10]{D}), so applying Goursat-Ribet's Theorem (see e.g.~\cite[Prop. 5.1]{Ch}), we deduce that the covers $U_{\ell,g}\ra W_{\ell,g}$ and $V_g\ra W_{\ell,g}$ are disjoint; in particular, if $V_{\ell,g}$ is the compositum of $V_g\ra W_{\ell,g}$ and $U_{\ell,g}\ra W_{\ell,g}$, then the Galois group of $V_{\ell,g}\ra W_{\ell,g}$ is that of $U_{\ell,g}\ra W_{\ell,g}$ which is precisely $\Omega(N,\F_\ell)$. In particular, as $\Omega(N,\F_\ell)$ has no non trivial abelian quotient, we deduce $\rho_\ell(\pi_1(\bar{V}_g,\bar{\mu}))=\Omega(N,\F_\ell)$.

\par
 \medskip
 The statement of linear disjointness follows by invoking once more Goursat-Ribet's Theorem.

\par
\medskip
 As far as the minimal value for $\ell$ is concerned, we exploit the remark following~\cite[Lemma $6.1$]{H}. Besides the fact that $\ell\geqslant 5$, there are two conditions needed for Hall's Theorem to apply: $\ell$ should not divide the multiplicity of any irreducible polynomial in $\F_q[t]$ appearing as a factor of the denominator of the $j$-invariant of $E$, and if $f(t)=(c-t)g(t)$ then the Galois group of the torsion field $K(E_f[\ell])$ over $K$ contains $SL(2,\F_\ell)$.
\par
 This last condition can be made explicit thanks to a result due to A. Cojocaru and C. Hall (see~\cite[Th. 1]{CH}) about the uniformity (in the function field case) of Serre's result on the surjectivity of mod $\ell$ Galois representations of elliptic curves. Their Theorem gives a very explicit lower bound for the first prime $\ell$ for which surjectivity holds, and this bound depends only on the genus of the function field over which the elliptic curve considered is defined. Applying their formula in our case (i.e. in the case of a rational function field), we obtain that the surjectivity property holds as soon as $\ell\geqslant 13$.

\end{proof}

 Now Corollary~\ref{up-bound-delta} may be applied in the framework of Section~\ref{quad-twist}, that means with $U=U_g$ and $\rho_\ell$ corresponding for each $\ell$ to the restriction of $\mt{T}_{d,\ell}$ to $U_g$. For a suitable choice of sieving sets $\Theta_\ell$ we might then obtain a proof of Theorem~\ref{estim_irred}, provided we can produce a good enough uniform lower bound for $H$. However we can do better than just investigating the irreducibility of $L$-functions when averaging over $U_g(\F_q)$. Indeed as noticed by Kowalski in~\cite[8.6]{KoLS} and as we did in~\cite{J} in the context of random walks on integral points of orthogonal groups, we can get quantitative information on the maximality of the Galois group of the $\Q$-polynomial $L_{\rm red}(E_c/K;T)$, when $c$ runs over the set of $\F_q$-rational points of $U_g$.
\par
 First, the functional equation~(\ref{eq_fonc}) imposes the Galois group of $L_{\rm red}(E_c/K;T)$ to be strictly smaller than the full symmetric group $\mathfrak{S}_{N_{\rm red}}$. Indeed that Galois group acts on pairs of roots $\{\alpha_i,\alpha_{N_{\rm red}/2}\}$ for $1\leqslant i\leqslant N_{\rm red}/2+i$. We deduce that the maximal Galois group that the $\Q$-polynomial $L_{\rm red}(E_c/K;T)$ can have is the group denoted $W_{N_{\rm red}}$. This group coincides with the Weyl group of the algebraic group ${\bf O}(N)$. That is an instance of a much more general principle according to which the Galois group of the characteristic polynomial of an integral point of any reductive group ${\bf G}/\Q$ should ``generically'' coincide with the Weyl group of ${\bf G}$ (at least if $\G/\Q$ is split). Such a question is investigated in~\cite{JKZ} in the case of a split model over $\Q$ of the exceptional group ${\bf E}_8$, in~\cite[Chap. 7]{KoLS} for ${\bf SL}(n)$ and ${\bf Sp}(2g)$ and in~\cite{J} for the orthogonal group with respect to an indefinite quadratic form.
\par
 As a consequence, it seems fair to say that the $\Q$-polynomial $L_{\rm red}(E_c/K;T)$ has \emph{small Galois group} if the Galois group of its splitting field over $\Q$ is strictly contained in $W_{N_{\rm red}}$. Now the generalized version of Theorem~\ref{estim_irred} we shall prove is the following:

 \begin{theorem} \label{estim_petitgalois}
   With notation as in Section~\ref{quant-katz}, let $L_{{\rm red},c}$ denote the reduced $L$-function of the quadratic twist $E_c$ of $E$. For $N\geqslant 5$, $d=\deg(g)+1\geqslant d_0(E)$, and $q\geqslant q_0(E)$, we have 
    $$
    |\{c\in \A^1(\F_q)\mid g(c)m(c)\not=0 \text{ and }  L_{{\rm red},c}\text{ has small Galois group }\}|\ll N^2|G(V_g/U_g)|q^{1-\gamma}\log q\,,
    $$
  where the implied constant depends only on $j(E)$ (e.g. it does not depend on $N$, and $q$ could be replaced by $q^n$ for any $n\geqslant 1$), and where we can choose $2\gamma^{-1}=7N^2-7N+4$.
   \end{theorem}

 Besides the large sieve inequality contained in the statement of Corollary~\ref{up-bound-delta}, the other ingredient we need in order to prove Theorem~\ref{estim_petitgalois} consists in a suitable choice of families $\Theta$ and, for each such family, a good enough uniform lower bound for the constant $H$. As the Galois groups we are investigating are subgroups of the Weyl group $W_{N_{\rm red}}$, the choice for the families $\Theta$ is almost the same as in our previous work~\cite[Lemma 20]{J}. The only thing that changes here is the number of left cosets of $O(N,\F_\ell)$, with respect to $\Omega(N,\F_\ell)$, that we consider. Indeed, in loc. cit., we were only working with matrices whose determinant was $1$, namely elements of $SO(n,m)(\Z)$. In that case the spinor norm is the only invariant that discriminates a left coset from another. Here we also have to handle the cosets given by $(\det=-1, \nsp=\pm 1)$. To describe the families $\Theta$ we choose, we need first to recall a few notations from~\cite{J}: we consider the set of polynomials
 $$
 M_{N,\ell}=\Bigl\{1+b_1T+\cdots+b_NT^N\mid b_i\in\F_{\ell}, b_N^2=1 \,\, \text{and}\,\, b_{N-i}=b_Nb_i,\,\, \text{if}\,\,0\leqslant i\leqslant \lfloor N/2\rfloor\Bigr\}\,.
 $$
 
  For a fixed left coset $\alpha_\ell G_\ell^{g,V}\subset O(N,\F_\ell)$, the integer $N_{\rm red}$ is well defined as the degree of the reduced characteristic polynomial of any matrix in that coset in the sense of~(\ref{red-char}), namely 
 $$
 N_{\rm red}=\begin{cases}  &N \text{ if } N \text{ is even and } \det \alpha_\ell=1\,,\\
                             &N-2 \text{ if } N \text{ is even and } \det \alpha_\ell=-1\,,\\
                         &N-1 \text{ if } N \text{ is odd }. \end{cases}     
 $$

 Thus we can define a family of sieving sets $(\Theta_\ell)_\ell$ such that $\Theta_\ell\subset Y_\ell=\alpha_\ell G_\ell^{g,V}$, via the formula
  $$
\Theta_\ell=\{g\in O(N,\F_\ell)\mid (\det, \nsp)(g)=(\eps_\ell^{(1)},\eps_\ell^{(2)}),\,\det(1-Tg)_{\rm red}\in \tilde{\Theta}_\ell\}\,,
 $$
 where $\tilde{\Theta}_\ell$ is for each $\ell$ a fixed subset of $M_{N_{\rm red},\ell}$.
 
 \par
 \medskip
  Now our choice for families of sieving sets corresponds to the following sets of polynomials and values for $\eps_\ell^{(1)}, \eps_\ell^{(2)}$.

 \begin{enumerate}
 \item The set $\Theta_\ell^{(1)}$ corresponds to the set of polynomials $\tilde{\Theta}^{(1)}_\ell$
             \begin{itemize}\item which are either irreducible \emph{if N is odd or if $N$ is even and $\eps_\ell^{(1)}=-1$}, or irreducible with a fixed value modulo nonzero squares of $\F_\ell$ at $-1$ and satisfy $\disc(f)=\disc(Q)$ \emph{if $N$ is even, $\eps_\ell^{(1)}=1$, and $O(N_{\rm red},\F_\ell)=O(N,\F_\ell)$ is nonsplit},
                             \item which factor as a product of two distinct monic irreducible polynomials of degree $N_{\rm red}/2$ \emph{if $\eps_\ell^{(1)}=1$, $O(N_{\rm red},\F_\ell)$ is split and $\ell\equiv 1\,(\mathrm{mod}\,4)$},
                             \item which factor as a product of an irreducible monic quadratic polynomial and an irreducible polynomial of degree $N_{\rm red}-2$ \emph{if $\eps_\ell^{(1)}=1$, $O(N_{\rm red},\F_\ell)$ is split and $\ell\equiv 3\,(\mathrm{mod}\,4)$}.    
              \end{itemize}
  \item Let $\tilde{\Theta}_\ell^{(2)}$ be the set of polynomials $f$ in $M_{N_{\rm red},\ell}$ with a fixed value modulo nonzero squares of $\F_\ell$ at $-1$, which satisfy $\disc(f)=\disc(Q)$ and which factor as a product of a monic quadratic polynomial with distinct monic irreducible polynomials of odd degree.
  \item Let $\tilde{\Theta}_\ell^{(3)}$ be the set of polynomials $f$ in $M_{N_{\rm red},\ell}$ with a fixed value modulo nonzero squares of $\F_\ell$ at $-1$, which satisfy $\disc(f)=\disc(Q)$ and with associated polynomial $h$ (such that $f=x^nh(x+x^{-1})$) being separable with at least one factor of prime degree $>N_{\rm red}/4$.   
  \item Let $\tilde{\Theta}_\ell^{(4)}$ be the set of polynomials $f$ in $M_{N_{\rm red},\ell}$ with a fixed value modulo nonzero squares of $\F_\ell$ at $-1$, which satisfy $\disc(f)=\disc(Q)$ and with associated polynomial $h$ being separable with one irreducible quadratic factor and no other irreducible factor of even degree.                          
  \end{enumerate}
 
  The above choices are justified by~\cite[Lemma $7.1(iii)$]{KoCrelle} and the discussion following $(8.4)$ in loc. cit. It is indeed proven therein that if the four conjugacy classes of $W_{N_{\rm red}}$, that are implicitly emphasized by the four choices above, are contained in a subgroup of $W_{N_{\rm red}}$ then that subgroup is the full group $W_{N_{\rm red}}$. Notice here that finding a good family $\Theta^{(1)}$ is quite an issue, though its purpose is ``as simple as'' detecting the irreducibility of the $\Q$-polynomial we investigate (namely the reduced $L$-function of a quadratic twist of the elliptic curve $E$). We refer the reader to~\cite[Sections 2 and 3]{J} for details and explanations regarding this particular feature and to~\cite[end of Section 7]{Ka} for the trick of looking both at reductions of the type ($\F_\ell$-irreducible of degree $2$)($\F_\ell$-irreducible of degree $N_{\rm red}-2$) and ($\F_\ell$-irreducible of degree $N_{\rm red}/2$)(a different $\F_\ell$-irreducible of degree $N_{\rm red}/2$) to detect $\Q$-irreducibility.
  \par
  As explained earlier, we need uniform lower bounds (in terms of the degree $N$) for each constant $H$ corresponding to our different choices of families $\Theta$. Those bounds are essentially provided by the following Lemma:

  \begin{lemme}\label{estim-theta}
  Suppose $N\geqslant 5$ . For the first three families of sieving sets above, the following estimates hold for $\ell\geqslant\ell_0(N)$:
 $$ 
  \frac{|\Theta_\ell^{(1)}|}{|\Omega(N,\F_\ell)|}\geqslant \frac{1}{4N^2}\,,\indent
  \frac{|\Theta_\ell^{(2)}|}{|\Omega(N,\F_\ell)|}\geqslant \frac{1}{5N}\,,\indent
  \frac{|\Theta_\ell^{(3)}|}{|\Omega(N,\F_\ell)|}\geqslant \frac{7}{3N}\,.
 $$
For the family fo sieving sets $(\Theta_\ell^{(4)})$, we have, for $\ell\geqslant \ell_0(N)$,
$$
\frac{|\Theta_\ell^{(4)}|}{|\Omega(N,\F_\ell)|}\geqslant \frac{1}{9N(N-6)}\,\text{ if } N_{\rm red}\geqslant 10,\text{ and } \frac{|\Theta_\ell^{(4)}|}{|\Omega(N,\F_\ell)|}\geqslant \frac{1}{N^2},\text{otherwise}\,.
$$
  \end{lemme}

\begin{proof}
 For $\Theta_\ell^{(1)}$ the major part of the argument is contained in the proof of~\cite[Lemma 20]{J}. From the estimates proved in loc. cit., we easily deduce, in the case where $N=N_{\rm red}$ is even, $\eps_\ell^{(1)}=1$ and the orthogonal group $O(N,\F_\ell)$ is nonsplit, the following asymptotic expansion:
 $$
 \frac{|\Theta_\ell^{(1)}|}{|\Omega(N,\F_\ell)|}\geqslant \frac{1}{2N}-\frac{N(N-1)+4N+4}{4N\ell}+O\big(\frac{N^3}{\ell^2}\big)\,,
 $$
 with an absolute implied  constant. Thus
 $$
 \frac{|\Theta_\ell^{(1)}|}{|\Omega(N,\F_\ell)|}\geqslant \frac{1}{3N}\,,
 $$
 for $\ell\geqslant \ell_0(N)$. Moreover, the results of Katz~(\cite[Section 6]{Ka}) provide us with the following lower bounds:
 $$
 \frac{|\Theta_\ell^{(1)}|}{|\Omega(N,\F_\ell)|}\geqslant \frac{1}{4(N-2)}\,,
 $$
as soon as $\ell\geqslant \max(7,N/2-1)$, if $N$ is even and $\eps_\ell^{(1)}=-1$ (this is~\cite[Lemma 6.3]{Ka}). Other lemmas from loc. cit. are also used in~\cite[Lemma 20]{J} to prove
$$
 \frac{|\Theta_\ell^{(1)}|}{|\Omega(N,\F_\ell)|}\geqslant \frac{1}{4N^2}\,,
 $$
 in the split case and assuming $\eps_\ell^{(1)}=1$, and
 $$
 \frac{|\Theta_\ell^{(1)}|}{|\Omega(N,\F_\ell)|}\geqslant \frac{1}{2N-2}\,,
 $$
 as soon as $\ell\geqslant \max(7,(N-1)/2)$ if $N$ is odd. 
 \par
 Overall, we obtain the lower bound stated as soon as $N\geqslant 3$.
 
 \par
 \medskip
  As far as $\Theta_\ell^{(2)}$ is concerned, we deduce from the formula of~\cite[Proof of Lemma 20]{J} that
  $$
  \frac{|\Theta_\ell^{(2)}|}{|\Omega(N,\F_\ell)|}\geqslant \frac{1}{4N}-\frac{3N^2+3N+1}{24N\ell}+O\bigl(\frac{N^3}{\ell^2}\bigr)\,,
  $$
with an absolute implied constant. So for $\ell\geqslant \ell_0(N)$,
$$
\frac{|\Theta_\ell^{(2)}|}{|\Omega(N,\F_\ell)|}\geqslant \frac{1}{5N}\,.
$$

 Finally for $\Theta^{(3)}$ and $\Theta^{(4)}$, loc. cit. provides us on the one hand with:
 $$
 \frac{|\Theta_\ell^{(3)}|}{|\Omega(N,\F_\ell)|}\geqslant \frac{7}{2N}-\frac{7N^2+23N+28}{4N\ell}+O\bigl(\frac{N^3}{\ell^2}\bigr)\,,
 $$
as soon as $N\geqslant 5$ and with an absolute implied constant, and on the other hand, for $N_{\rm red}\geqslant 10$, with
 $$
 \frac{|\Theta_\ell^{(4)}|}{|\Omega(N,\F_\ell)|}\geqslant \frac{1}{8(N-6)}-\frac{5N(N-1)+1}{80(N-6)\ell}+O\bigl(\frac{N^3}{\ell^2}\bigr)\,,
 $$
with an absolute implied constant. The last inequality is obtained by combining the expression for $|\tilde{\Theta}_\ell^{(4)}|$ in~\cite[Proof of Lemma 20]{J} and Lemma $16$ of loc. cit. From the above expressions we deduce:
 $$
 \frac{|\Theta_\ell^{(3)}|}{|\Omega(N,\F_\ell)|}\geqslant \frac{7}{3N}\,,\indent \frac{|\Theta_\ell^{(4)}|}{|\Omega(N,\F_\ell)|}\geqslant \frac{1}{9(N-6)}\,,
 $$
 for $\ell\geqslant \ell_0(N)$, with the second inequality holding if we assume $N_{\rm red}\geqslant 10$. To conclude, we need to handle separately the case where $N_{\rm red}\leqslant 8$ (i.e. $N_{\rm red}\in\{4,6,8\}$). To that end we use once more the corresponding formul\ae\, contained in the proof of~\cite[Lemma 20]{J} as well as~\cite[Lemma 16]{J}. It is then straightforward to derive, for $\ell\geqslant \ell_0(N)$,
$$
|\Theta_\ell^{(4)}||\Omega(N,\F_\ell)|^{-1}\geqslant \frac{1}{N^2}\,,
$$
from these references.
\end{proof}

 \begin{remarks}
  $(i)$ From the lower bound for $|\Theta_\ell^{(1)}||\Omega(N,\F_\ell)|^{-1}$ in Lemma~\ref{estim-theta}, we deduce a new proof of $(2)$ of~\cite[Lemma 5.1]{KoTwist}. In loc. cit. the author suggests that such a result should follow from an adaptation to the orthogonal case of the method used by Chavdarov in~\cite[Th. 3.5 attributed to Borel]{Ch} to prove that the matrices in $Sp(2g,\F_\ell)$ are equidistributed in the subsets containing matrices having a common characteristic polynomial. As explained in~\cite[Section 2]{J}, such an adaptation is not straightforward at all (e.g. the analog equidistribution statement does not hold in full generality) because of the many complications due to the lack of good topological properties of the orthogonal group as an algebraic group.
  \par $(ii)$ To get a deeper understanding of the issue discussed in $(i)$, we easily see that the crucial question is: given a quadratic (resp. symplectic) $2n$-dimensional space $(V,Q)$ over $\F_\ell$ and a polynomial $f$ satisfying~(\ref{eq_fonc}) (resp. a self reciprocal $f$), is it true that there is an $M\in O(2n,\F_\ell)$ (resp. $M\in Sp(2n,\F_\ell)$) such that $\det(1-TM)=f$?
  \par
  For simplicity we assume that we always choose a separable $f$ and, in the orthogonal case, that $f$ is monic. In that case we are looking for a matrix in $SO(2n,\F_\ell)$. The strategy used by Chavdarov is to consider this question as a rationality problem: does the set of $\overline{\F_\ell}$-points in ${\bf SO}(2n)$ (resp. ${\bf Sp}(2n)$) with common characteristic polynomial $f$ contain an $\F_\ell$-point? Chavdarov considers the centralizer $C_{\bar{M}}$ under $Sp(2n,\overline{\F_\ell})$ of an element $\bar{M}\in Sp(2n,\overline{\F_\ell})$ with characteristic polynomial $f$. Because ${\bf Sp}(2n)$ is simply connected $C_{\bar{M}}$ is connected and thus is equipped with Lang's isogeny. This observation yields an $\F_\ell$-point with characteristic polynomial $f$ in the symplectic case. As ${\bf SO}(2n)$ is not simply connected that method cannot be directly applied in the orthogonal case. Indeed a result of Steinberg (see~\cite[II Cor. 4.4]{SpSt}) asserts that non connected centralizers will occur as soon as the group acting is not simply connected. Moreover we cannot ``lift'' that method to the simply connected cover ${\bf Spin}(2n)$ of ${\bf SO}(2n)$ because the possibilty to produce an element in ${\rm Spin}(2n,\F_\ell)$ from an isometry in $SO(2n,\F_\ell)$ depends on the spinor norm of the isometry (see~\cite[Rem. 6.1]{Ka}).
  \par
  The spinor norm turns out to be the right invariant that enables us to answer the question in the orthogonal case. First a result of Zassenhaus (see~\cite[(2.1) and p. 446]{Za}) asserts that the spinor of an isometry with characteristic polynomial $f$ has spinor norm $f(-1)$ modulo squares (provided $-1$ is not a root of $f$). Second, an easy computation yields $\disc(f)=(-1)^nf(1)f(-1)$ (see~\cite[Th. 1 and Th. 2]{E}). Finally, a result of Baeza (see~\cite[prop. 3.6 and Th. 3.7]{Ba} and also~\cite[prop. A.3]{GM}) asserts that an $M\in SO(2n,\F_\ell)$ with prescribed characteristic polynomial $f$ exists if and only if $\disc(Q)=\disc(f)$ (modulo squares).
  \par
  The above condition on equality of discriminants can be restated in the framework of Rodriguez-Villegas' unpublished note~\cite{R}: indeed in loc. cit., the author gives a very general and intrinseque construction that relates to the present question. If we restrict to the case of finite fields, that construction implies the following: let $\eps=\pm 1$ and consider two coprime polynomials $f,g\in\F_\ell[T]$ with degree $2n$ such that $f$ is self reciprocal and $T^{2n}g(T^{-1})=-\eps g(T)$. Then one can produce a $2n$-dimensional quadratic (resp. symplectic) space if $\eps=1$ (resp. $\eps=-1$) and two elements $A,B$ in $O(2n,\F_\ell)$ (resp. $Sp(2n,\F_\ell)$) with respective characteristic polynomials $f$ and $g$. Rodriguez-Villegas shows that if $\eps=-1$ and given a polynomial $f$ as above, we can always produce a suitable polynomial $g$, reproving in turn Chavdarov's result~\cite[Lemma 3.4]{Ch}. Now if $\eps=1$ the quadratic space we end up with has discriminant equal to the resultant ${\rm Res}(f,g)$ of $f$ and $g$. Combining that with the facts mentioned above we deduce that given $f$, finding a suitable $g$ is equivalent to solving  the equation ${\rm Res}(f,g)=\disc(f)$ for $g$.

 \par $(iii)$ The smallest prime $\ell_0(N)$ for which the estimates of Lemma~\ref{estim-theta} hold can be made explicit. Indeed, exploiting the asymptotic expansions we used in the proof, we deduce by inspection that all the inequalities hold as soon as $\ell\geqslant 5N^2$. %However when putting things together in order to prove Theorem~\ref{estim_petitgalois}, we will see that the implied constant can be chosen in such a way that the above restriction on $\ell$ does not have any consequence on the range of validity of the final result
 \par $(iv)$ It is very likely (though the author could not come up with a satisfactory statement covering all the different cases) that a clever adaptation and/or combination of Katz's computations~\cite[Lemmas 6.2 through 6.6]{Ka} would yield similar estimates as those given in Lemma~\ref{estim-theta}. Note that, in performing such a computation, it might be more convenient to use the criterion~\cite[Lemma 7.1\emph{(ii)}]{KoCrelle} rather than~\cite[Lemma 7.1\emph{(iii)}]{KoCrelle} for a subgroup of $W_{N_{\rm red}}$ to be equal to the full group.
 \end{remarks}

 To each of the families $\Theta^{(i)}$ of Lemma~\ref{estim-theta}, where $1\leqslant i\leqslant 4$, corresponds a conjugacy class, say $\Theta^{\sharp(i)}$, of $W_{N_{\rm red}}$. So with notation as in Theorem~\ref{estim_petitgalois}, we have
 $$
 |\{c\mid L_{{\rm red},c} \text{ has small Galois group }\}|\leqslant \sum_{1\leqslant i\leqslant 4}{|\{c\mid \Gal(L_{{\rm red},c}/\Q)\cap\Theta^{\sharp(i)}=\emptyset\}|}\,,
 $$
 where the parameter $c$ runs over the elements of $\A^1(\F_q)$ such that $f(t)=(c-t)g(t)\in F_d(\F_q)$. The explanations following~\cite[(8.4)]{KoCrelle} and justifying our choice of families $\Theta^{(i)}$ enable us to deduce:
 $$
 |\{c\mid L_{{\rm red},c} \text{ has small Galois group }\}|\leqslant \sum_{1\leqslant i\leqslant 4}{|\{u\in U_g(\F_q)\mid \rho_\ell(\Fr_u)\not\in\Theta_\ell^{(i)}\text{ for } \ell_0(E)\leqslant\ell\leqslant L\}|}\,,
 $$
 where $\ell_0(E)$ is chosen in such a way that both Lemma~\ref{lemme-chris} and Lemma~\ref{estim-theta} hold and $L\geqslant \ell_0(E)$ is a fixed parameter (in particular, because of point $(ii)$ of the above remark, we can use the inequalities of Lemma~\ref{estim-theta} provided $L\geqslant 5N^2$).
 
 \par
 \medskip
 Set
 $$
 \Lambda=\{\ell\text{ prime}\mid \ell_0(E)\leqslant \ell\leqslant L\}\,.
 $$
 
   Applying Corollary~\ref{up-bound-delta}, Lemma~\ref{estim-theta} and the Prime Number Theorem, we get
  $$
  |\{c\mid L_{{\rm red},c} \text{ has small Galois group }\}|\ll N^2|G(V_g/U_g)| L^{-1}\log L (q+Cq^{1/2}(L+1)^A)\,,
  $$
  with an implied constant depending only on $j(E)$.

  \begin{remark}
   Note that, when applying the Prime Number Theorem, a natural constraint is imposed on $L$: since we only consider primes greater than $5N^2$, we need to choose (roughly) $L\gg N^2\log(N)$ (see~\cite[end of Section 8]{KoCrelle}). As explained in loc. cit. this condition can in fact be removed by modifying appropriately the implied constant in the above inequality. Indeed, that inequality becomes trivial as soon as $N^2\geqslant L(\log L)^{-1}$. 
  
  \end{remark}

  \par
  Setting $L=q^{1/2A}$ we deduce
 $$
 |\{c\mid L_{{\rm red},c} \text{ has small Galois group }\}|\ll N^2 |G(V_g/U_g)|q^{1-1/2A}\log q\,,
 $$
 with an implied constant depending only on $j(E)$.
 
 \par
 \medskip
  Looking back at the definition of $A$ (see Corollary~\ref{up-bound-delta}), we see that 
  $$
  2A=\frac{7}{2}N^2-\frac{7}{2}N+2\,,
  $$
  so that the proof of Theorem~\ref{estim_petitgalois} is complete.

\begin{remark}
 Note that we need to assume that the common degree $d$ of the twisting polynomials is large enough in order to apply Corollary~\ref{up-bound-delta}. The degree $N$ of the $L$-functions we study is related to $d$ (sometimes in a very explicit way, see the following section), as~(\ref{N-M-A}) shows quite obviously. As a consequence, we need to have $L=q^{1/2A}$ big enough (i.e. greater than the parameter $\ell_0(E)$ of Lemma~\ref{lemme-chris}) and simultaneously $d$ (and thus $N$) big enough. This is of course always possible if $q$ is big enough and this is why we assume $q\geqslant q_0(E)$ in the statement of Theorem~\ref{quant-katz} and Theorem~\ref{estim_petitgalois}. This might be a bit unsatisfactory as the smallest suitable value of $q$ is not clear, however we can also usefully notice that the estimate of Theorem~\ref{estim_petitgalois} is trivial for small values of $q$ (say if $N^2|G(V_g/U_g)|\gg q^{1/2A}\log(q)^{-1}$). In the next subsection though, we see that in the case where $E$ is a Legendre curve, a much more precise answer to that question arises from the resolution of the uniformity issue.
\end{remark}

\subsection{Uniformity in the case of a Legendre curve} 

 In this last subsection we focus on the case where $E/K$ is a Legendre curve. Keeping the notation $K=\F_q(t)$ we let $E$ be the elliptic curve over $K$ defined as the projective compactification of the affine curve given by
$$
y^2=x(x-1)(x-t)\,.
$$
 
The affine variety $F_d$ can be seen in that case as the set of polynomials with degree $d$ in $\overline{\F_q}[t]$ that are coprime to $t(t-1)$. Lemma $6.7$ of~\cite{H} provides us on the one hand with the dimension $N$ of the space $\mt{V}_{f,\ell}$:
\begin{equation} \label{dimV}
N=\begin{cases}
&2d\text{ if } d\text{ is even}\,,\\
&2d-1 \text{ if } d\text{ is odd}\,,
  \end{cases}
\end{equation}
and on the other hand with the fact that Lemma~\ref{lemme-chris} holds for any $d\geqslant 2$ and any $\ell\geqslant 5$. So if we fix a polynomial $g\in F_{d-1}$ and a Galois cover $V_g\ra U_g$ satisfying the hypotheses of Lemma~\ref{lemme-chris}, we can state the following result which can be seen as a uniform quantitative version of Katz's Theorem in the special case (stated in the introduction) of twists of Legendre curves:

\begin{theorem} \label{petit-galois-legendre} 
Let $E/K$ be the Legendre curve defined as above. With the same notation as in Theorem~\ref{estim_petitgalois} we have, for any $d\geqslant 3$ and any power $q$ of $p$,
 $$
|\{c\in \A^1(\F_q)\mid g(c)\not=0, c\not=0,1 \text{ and }  L_{{\rm red},c}\text{ has small Galois group }\}|\ll d^2|G(V_g/U_g)|q^{1-\gamma}\log q\,,
 $$
 with an absolute implied constant and where we can choose $2\gamma^{-1}=7N^2-7N+4$ (the link between $N$ and $d$ being given by~(\ref{dimV})).
\end{theorem}

\begin{proof}
 The statement follows directly from the combination of Theorem~\ref{estim_petitgalois}, Lemma $6.7$ in~\cite{H} and the fact that the denominator of the $j$-invaraint of the Legendre curve $E$ does not have any irreducible factor appearing with multiplicity greater than $2$. 
\end{proof}

\end{document}